\documentclass[12pt]{amsart}
\usepackage{color}
\usepackage{amssymb}
\usepackage{comment}
\usepackage{pdfpages}
\usepackage{graphicx}% Include figure files
\usepackage{dcolumn}% Align table columns on decimal point
\RequirePackage[numbers]{natbib}
\usepackage{hyperref}
\usepackage{paralist}

\usepackage{tikz} %for pictures
\usepackage{dcolumn}
\usetikzlibrary{arrows,positioning}
\tikzset{
    >=stealth',
    pil/.style={
           ->,
           thick,
           shorten <=2pt,
           shorten >=2pt,}
}

\usepackage{graphicx}
\graphicspath{{./figures/}}
\usepackage{amsfonts}
\usepackage{amsmath}
\usepackage{amsthm}
\usepackage{amssymb}
\usepackage{amsbsy}
\usepackage{epsfig}
\usepackage{fullpage}
\usepackage{natbib, mathrsfs} 
\usepackage{verbatim}
\usepackage[latin1]{inputenc}
\usepackage{mhequ}
\usepackage{algorithm}
\usepackage{algorithmic}
\usepackage{bbm}

\numberwithin{equation}{section}

\def \be{\begin{equs}}
\def \ee{\end{equs}}
\def \P{\mathbb{P}}
\def \E{\mathbb{E}}

\def \gd{G}

\def \bd{B}
\def \cov{W}

\newtheorem{theorem}{Theorem}[section]
\newtheorem{lemma}[theorem]{Lemma}

\newtheorem{prop}[theorem]{Proposition}
\newtheorem{assumptions}[theorem]{Assumptions}
\newtheorem{assumption}[theorem]{Assumption}

\theoremstyle{plain}
\newtheorem{thm}{Theorem}
\newtheorem*{thm-non}{Theorem}
\newtheorem{cor}[thm]{Corollary}

\theoremstyle{definition}
\newtheorem{defn}[theorem]{Definition}
\newtheorem{remark}[theorem]{Remark}
\begin{document}

%\begin{frontmatter}
% "Title of the paper"
\title[HMC on Multimodal Densities]{Does Hamiltonian Monte Carlo mix faster than a random walk on multimodal densities?}

% indicate corresponding author with \corref{}
% \author{\fnms{John} \snm{Smith}\corref{}\ead[label=e1]{smith@foo.com}\thanksref{t1}}
 %\thankstext{t1}{Thanks to somebody} 
 %\address{line 1\\ line 2\\ printead{e1}}
% \affiliation{Some University}

\author{Oren Mangoubi$^{\flat}$}
\thanks{$^{\flat}$oren.mangoubi@gmail.com, 
\'{E}cole Polytechnique F\'{e}d\'{e}rale de Lausanne (EPFL), 
IC IINFCOM THL3,
1015 Lausanne, Switzerland}

\author{Natesh S. Pillai$^{\ddag}$}
\thanks{$^{\ddag}$pillai@fas.harvard.edu, 
   Department of Statistics,
    Harvard University, 1 Oxford Street, Cambridge
    MA 02138, USA}

\author{Aaron Smith$^{\sharp}$}
\thanks{$^{\sharp}$smith.aaron.matthew@gmail.com, 
   Department of Mathematics and Statistics,
University of Ottawa, 585 King Edward Avenue, Ottawa
ON K1N 7N5, Canada}

 \thanks{NSP gratefully acknowledges the support of ONR.   AS was supported by an NSERC Discovery grant.  OM was supported by NSF-DMS 1312831 and by an NSERC Discovery grant.}

\maketitle
% AMS subject classifications (used in AMS journals)
   %\subjclass{Primary 60J22; Secondary 60J20, 60H15, 65C40}

   % AMS keywords (used in AMS journals)
 % \keywords{Approximate MCMC, Austerity framework, Stochastic Gradient Langevin, ABC-MCMC}

%\begin{keyword}[class=AMS]
%\kwd[Primary ]{1232}
%\kwd{ser}
%\kwd[; secondary ]{123}
%\end{keyword}

%\begin{keyword}
%\kwd{}
%\kwd{}
%\end{keyword}

%\end{frontmatter}

% AOS,AOAS: If there are supplements please fill:
%\begin{supplement}[id=suppA]
%  \sname{Supplement A}
%  \stitle{Title}
%  \slink[url]{http://lib.stat.cmu.edu/aoas/???/???}
%  \sdescription{Some text}
%\end{supplement}

%\tableofcontents

\begin{abstract}
Hamiltonian Monte Carlo (HMC) is a very popular and generic collection of Markov chain Monte Carlo (MCMC) algorithms. One explanation for the popularity of HMC algorithms is their excellent performance as the dimension $d$ of the target becomes large: under conditions that are satisfied for many common statistical models, optimally-tuned HMC algorithms have a running time that scales like $d^{0.25}$. In stark contrast, the running time of the usual Random-Walk Metropolis (RWM) algorithm, optimally tuned, scales like $d$. This superior scaling of the HMC algorithm with dimension is attributed to the fact that it, unlike RWM, incorporates the gradient information in the proposal distribution. In this paper, we investigate a different scaling question: does HMC beat RWM for highly \textit{multimodal} targets? We find that the answer is often \textit{no}. We compute the spectral gaps for both the algorithms for a specific class of multimodal target densities, and show that they are identical. The key reason is that, within one mode, the gradient is effectively ignorant about other modes, thus negating the advantage the HMC algorithm enjoys in unimodal targets. We also give heuristic arguments suggesting that the above observation may hold quite generally. Our main tool for answering this question is a novel simple formula for the conductance of HMC using Liouville's theorem. This result allows us to compute the spectral gap of HMC algorithms, for both the classical HMC with isotropic momentum and the recent Riemannian HMC, for multimodal targets. 
\end{abstract}

\section{Introduction}
Markov chain Monte Carlo (MCMC) algorithms, and in particular the Hamiltonian Monte Carlo (HMC) algorithms, are workhorses in many scientific fields including physics \cite{Hybrid_MCMC}, statistics  and machine learning \cite{NUTS,Riemannian_HMC2, MCMC_Application_Machine_Learning,welling2011bayesian, mangoubi2018convex}, and molecular biology \cite{MCMC_Application_molecular_biology}. An important question practitioners face is how to choose an algorithm for a \textit{particular} problem? Are there any general principles that allow users to prefer one algorithm to another for \textit{generic} but \textit{simple} problems? One of the most fruitful branches of MCMC theory has focused on analyzing MCMC algorithms by studying their \textit{scaling limits.} The idea is to study how an MCMC algorithm's speed depends on some underlying parameter, often the dimension $d$ of the target distribution. Results from the theory of scaling limits suggest that HMC often has a running time of  $O(d^{0.25})$ \cite{beskos2013optimal} whereas  the Random-Walk Metropolis (RWM) algorithm has a much longer running time of $O(d)$ steps under very broad conditions \cite{roberts1997weak,mps12,bed07}. These results and other evidence, both empirical and theoretical (see \textit{e.g.,}  \cite{betancourt2017conceptual,mangoubi2017rapidp1,rabee2018HMCcoup, mangoubi2016rapid}) have lead to the widespread understanding that HMC is superior to RWM for a wide variety of problems. In this paper, we investigate the extent to which this superiority holds in another natural scaling regime: highly multimodal target distributions. There is no a priori reason to believe that algorithms superior in one regime are also superior in other regimes.

 To give a simple example, consider the mixture of Gaussians
\be \label{eqn:gausmix}
\frac{1}{2} \mathcal{N}(-1,\sigma^2) + \frac{1}{2} \mathcal{N}(1,\sigma^2).
\ee 
A natural regime to study is when $\sigma \rightarrow 0$. Let us tune both algorithms to optimize their performance within a single mode. To this end, for the RWM algorithm, results of \cite{roberts1997weak} (and also simple scale-invariance of the mixture components) imply that the proposal variance must be $O(\sigma^2)$. For the classical HMC algorithm  \cite{Hybrid_MCMC} with isotropic momentum, we set the integration length to be $O(\sigma)$; this is the first time we would expect the algorithm to take a U-turn (and is also suggested by scale-invariance of the mixture components). For the above example, with these tuning parameters, Theorem \ref{ThmHmcMultimodal} of this paper and Theorem 3 of our companion paper \cite{mangoubi2018simple} imply that the spectral gap of HMC and RWM \emph{both} decay exactly like $e^{-\frac{1}{2} \sigma^{-2}}$  as $\sigma \rightarrow 0$. Of course, we expect that the spectral gap goes to 0 quickly in both cases, just as it does in the high-dimensional scaling regime of \cite{roberts1997weak,beskos2013optimal}. The interesting fact is that they go to 0 at the \textit{same rate}, so that the asymptotic performance of the two algorithms is the same. For this example, instead of isotropic momentum, if we choose the metric to be the inverse of the Fisher information as in the Reimannian HMC of \cite{Riemannian_HMC2}, the spectral gap still decays like $e^{-\frac{1}{2} \sigma^{-2}}$  as $\sigma \rightarrow 0$.  Thus, our main conclusion is that, for this example is that HMC is \emph{not} better than RWM. 
Finally, the above results also hold for the high dimensional analog of \eqref{eqn:gausmix}; see Section \ref{sec:highdim}.

It is natural to ask if this comparison is a result of bad tuning. In fact, this is not the case. Tuning the integration length of HMC cannot improve its relative performance. The natural tuning for HMC used above is essentially optimal (in the sense of maximizing the effective sample size per unit computation; see Section \ref{sec:RemChoiceInt} for a more detailed discussion of tuning parameters and optimality). We give a further discussion of scaling results, as well as related open questions, in Section \ref{SecCompMulti}. We highlight here one fact from  Section \ref{SecCompMulti} that is particularly interesting for our simple target distribution \eqref{eqn:gausmix}: considering other tuning parameters makes HMC look even worse than RWM. More precisely, the computational cost of the HMC algorithm is \textit{always} at least on the order of $e^{-\frac{1}{2} \sigma^{-2}}$ even for much longer integration times, while well-chosen tuning parameters can make RWM much more efficient (again detailed in Section  \ref{sec:RemChoiceInt} and Section \ref{SecCompMulti}).

Readers and practitioners familiar with the MCMC literature may be surprised at the tone of the above paragraph, since HMC is generally viewed as having much better performance than many older and simpler random walk based MCMC algorithms  \cite{betancourt2017conceptual,roberts1997weak,beskos2013optimal,mangoubi2017concave,rabee2018HMCcoup}. However, the main heuristics justifying the dominance of HMC in the high-dimensional regime do not apply in our highly-multimodal regime. In particular, HMC can use knowledge of the gradient of the log-likelihood to make larger moves than RWM in the high-dimensional regime, but the gradient cannot detect multimodality. 

\subsection{Conductance Results}
One of the main tools we use in our calculations is \textit{Cheeger's inequality}. This well known tool provides a coarse estimate of the efficiency of an MCMC algorithm in terms of a geometric quantity called the \textit{conductance} (see\cite{cheeger1970lower, lawler1988bounds} and a survey of variants \cite{montenegro2006mathematical}). It is useful primarily because it can be easier to estimate than more precise measures of efficiency. 

Unfortunately, the conductance can still be difficult to compute. One of the main contributions of this paper is a simple exact formula for the conductance of both HMC and RHMC (see Theorem \ref{thm:1}). This formula relates the conductance of a set $S$ to the integral of a function, which we interpret as its ``crossing rate," on the boundary $\partial S$ of $S$.  For the `standard' HMC algorithm that uses isotropic momentum, Theorem \ref{thm:1} (see Corollary \ref{cor:1}) yields the following bound for conductance of a set $S$:
\be 
\mathrm{Conductance}(S) \leq \frac{1}{2}T { \int_{\partial S} \pi(q) \mathrm{d}q \over \pi(S)},
\ee 
where $T$ is the integration length and the integral is over the boundary $\partial S$ of the set $S$ and $\pi$ is the target density. In particular, the above formula
yields that the conductance can increase \emph{at most} linearly with the integration time for the HMC algorithm. This is also true for RHMC. In Section \ref{sec:RemChoiceInt}, we use this observation to show that changing the integration length cannot vastly improve the performance of the HMC algorithm in multimodal regimes.

 Markov chains do not typically have conductance formulas that can easily be expressed in terms of surface integrals. As seen from above, our new formula is much easier to use.  Having a simple formula for HMC is particularly useful because HMC algorithms are some of the most widely-used \cite{Riemannian_HMC2, cheung2009bayesian, mehlig1992hybrid} modern MCMC algorithms, but their theoretical properties are not yet well-understood (though again see \cite{holmes2014curvature, livingstone2016geometric, mangoubi2017concave, rabee2018HMCcoup}). Understanding the conductance of HMC algorithms can help us understand when these algorithms perform poorly, and often suggest what has gone wrong. 

As with \textit{e.g.,} \cite{holmes2014curvature}, we begin our quantitative study of mixing times by studying ``idealized" versions of HMC. That is, we ignore the error introduced by solving Hamilton's equations with a numerical integrator such as the leapfrog integrator and instead assume that we can solve the Hamilton's ODEs exactly. We expect similar qualitative conclusions to hold for appropriate practical implementations of HMC, and there is substantial work on relating ``ideal" Monte Carlo schemes to their numerical implementations (see \textit{e.g.,}  \cite{durmus2016sampling2} \cite{livingstone2016geometric}, \cite{mangoubi2017rapidp1}, \cite{mangoubi2017rapidp2}, \cite{lee2017convergence} \cite{mangoubi2018dimensionally}).

\subsection*{When can HMC beat RWM for multimodal targets?}
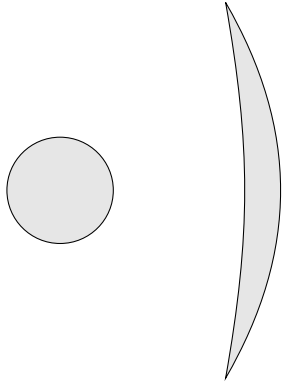
\begin{figure}
\begin{tikzpicture}[>=latex]
\coordinate (A) at (0,0) ;
\coordinate (B) at (0,-5) ;
\coordinate (C) at (4,0) ;
\coordinate (D) at (4,-5) ;

\filldraw[fill=gray!20]
 (A) to[out=-100,in=100] 
  (A)
  % the banana to the right
 (C) to[out=-80,in=80] 
   coordinate[pos=0.05] (auxru)
   coordinate[pos=0.95] (auxrl)
 (D) to[out=60,in=-60] 
 (C);
\node[circle,draw = black, fill = gray!20,inner sep=0.5cm] 
  at (1.8,-2.5) (circle) {};    
  \end{tikzpicture}
  \caption{We see a small deep mode next to a `banana' shaped mode. HMC can often explore long, skinny modes much more quickly than RWM. Thus, it is possible to tune the length of the banana shaped mode, in relation to the distance between the centers of the banana and the circle, so that the time for RWM to mix on the long mode is much larger than the time to escape the small mode, while the time for HMC to mix on the long mode is much smaller than the time to escape the small mode. In this case, the HMC algorithm can exhibit metastability while the RWM algorithm does not; we would expect HMC to mix more quickly than RWM in this situation.} \label{FigBanana}
  \end{figure}

We envision that there may be some classes of highly multimodal target densities where the HMC algorithm is superior to random walk based algorithms. One such example is shown in Figure \ref{FigBanana}. The basic idea is that HMC can often explore long, skinny modes much more quickly than RWM. If we consider an example with several modes, at least one of which is very long and skinny, the running time for  RWM may be determined by the time it takes to traverse the long and skinny mode, whereas the running time for HMC will be determined by the time it takes to travel between the modes. This `energy-entropy' competition could lead to different performances of the two algorithms, and certainly merits further study. In particular, the HMC algorithm in this example could exhibit `metastability' while the RWM algorithm does not. In our companion paper \cite{mangoubi2018simple}, we derive simple conditions for checking the metastability of commonly used MCMC algorithms.

It is also natural to ask if any of the popular HMC variants can substantially improve the performance of HMC on multimodal targets. We believe that this is an important research question, but there are many HMC variants and a careful survey of their performance is beyond the scope of the present paper; instead we give a quick summary in relation to our results. Our main results do not apply as stated to the popular NUTS algorithm \cite{NUTS}, though we expect very similar bounds to be true. Our main conductance bounds, Theorem \ref{thm:1}  and Corollary \ref{cor:1}, do apply as stated to the Riemannian HMC (RHMC) algorithm of \cite{Riemannian_HMC2}. For the toy example \eqref{eqn:gausmix}, these bounds can be used to show that the RHMC algorithm with its ``usual" tuning (the inverse of the Fisher information as suggested in \cite{Riemannian_HMC2}) does not offer substantial improvements on the RWM algorithm with the ``usual" tuning discussed above. On the other hand, it does appear possible to use extremely unusual tuning parameters for RHMC to substantially improve performance, just as this is possible for RWM (see Section \ref{SecHmcImp} for detailed discussion of this tuning for RWM). As with RWM, there does not seem any obvious way to ``guess" these good tuning parameters without unrealistically good knowledge of the locations of the modes, and so it is not clear if the existence of good tuning parameters has any practical importance. 
%We believe that an important (but perhaps difficult) research problem is to write down a precise theorem that encapsulates the idea that ``it is hard to guess tuning parameters" (though see \textit{e.g.} \cite{bhatnagar2010computational} and related work in the computer science literature ).

\subsection{Guide to Paper}

In Section \ref{SecAlgDefs}, we review basic notation and the definitions of the MCMC algorithms that we study in this paper: an ``ideal" version of the standard and Riemannian HMC algorithms, as well as a simple Metropolis-Hastings algorithm for comparison. We then prove bounds on the conductance for HMC in Section \ref{sec:HMC}. Finally, we illustrate the uses of our results by analyzing the performance of HMC for two simple but important examples in Section \ref{SecAppl}: a mixture of Gaussians and a highly degenerate Gaussian. Finally, we discuss the consequences of our work and give related conjectures in Section \ref{SecCompMulti}. The longer proofs are deferred to the appendices.

\subsection{Basic Notation}
We review the standard ``big-O" notation, which is used throughout the paper. For two nonnegative functions or sequences $f,g$, we write $f = O(g)$ as shorthand for the statement: there exist constants $0 < C_{1},C_{2} < \infty$ so that for all $x > C_{1}$, we have $f(x) \leq C_{2} \, g(x)$. We write $f = \Omega(g)$ for $g = O(f)$, and we write $f = \Theta(g)$ if both $f= O(g)$ and $g=O(f)$. Relatedly, we write $f = o(g)$ as shorthand for the statement: $\lim_{x \rightarrow \infty} \frac{f(x)}{g(x)} = 0$.  We write $f = \tilde{O}(g)$ if there exist constants $0 < C_{1},C_{2}, C_{3} < \infty$ so that for all $x > C_{1}$, we have $f(x) \leq C_{2} \, g(x) \log(x)^{C_{3}}$, and write $f = \tilde{\Omega}(g)$ for $g = \tilde{O}(f)$. Finally, we say that a function $f$ is ``bounded by a polynomial" if there exists $0 < c < \infty$ so that $f(x) = O(x^{c})$. 
\section{Algorithms and Notation} \label{SecAlgDefs}

In this section we review two important Hamiltonian Monte Carlo algorithms, as well as the commonly-used Random Walk Metropolis algorithm.  We also review some important definitions for MCMC algorithms, including careful definitions of the conductance, spectral gap and Cheeger inequality. 

\subsection{Basic Notation}
Throughout the remainder of the paper, we denote by $\pi$ the smooth density function of a probability distribution on $\mathbb{R}^{d}$. We denote by $\mathcal{L}(X)$ the distribution of a random variable $X$. Similarly, if $\mu$ is a probability measure, we write ``$X \sim \mu$" for ``$X$ has distribution $\mu$."  Throughout, we will generically let $Q \sim \pi$ and $P \sim \mathcal{N}(0, \mathrm{Id})$ be independent random variables, where $\mathrm{Id}$ is the $d$-dimensional identity matrix.

\subsection{Random Walk Metropolis}
The Random Walk Metropolis (RWM) algorithm (Algorithm \ref{alg:RWM}) is the most basic commonly-used MCMC algorithm.  At each step $i$ of the Markov chain, the RWM algorithm proposes to take the next step $\hat{x}_{i+1}$ in a random direction and distance from the current position $x_i$.  The step is accepted with a probability of $\mathrm{min}\{\frac{\pi(\hat{x}_{i+1})}{\pi(x_{i})}, 1\}$, in which case we set $x_{i+1} = \hat{x}_{i+1}$.  Otherwise, we say that the step is rejected and we set $x_{i+1} = x_{i}$, so that the algorithm stays at its current position until the next time step.

\begin{algorithm}[H]
\caption{Random Walk Metropolis \cite{metropolis1953equation}}\label{alg:RWM}
\flushleft
\textbf{Input:} Tuning parameter $\epsilon > 0$, starting point $x_0$, target density $\pi: \mathbb{R}^{d} \rightarrow [0, \infty)$. \\
\textbf{Output:} Markov chain $x_1, x_2, \ldots$, with stationary distribution $\pi$. \\
\begin{algorithmic}[1]
\FOR{$i = 1,2,\ldots$,}
\STATE Sample $p_{i} \sim \mathcal{N}(0,\epsilon^{2})^{d}$ and $U \sim \mathrm{Unif}([0,1])$. 
\STATE Set $\hat{x}_{i+1} = x_{i} + p_{i}$.
\IF{$U < \mathrm{min}\{\frac{\pi(\hat{x}_{i+1})}{\pi(x_i)}, 1\},$}
\STATE Set $x_{i+1} = \hat{x}_{i+1}$.
\ELSE  
\STATE Set $x_{i+1} = x_{i}$.
\ENDIF
\ENDFOR
\end{algorithmic}
\end{algorithm}

\subsection{Hamiltonian Monte Carlo algorithms}

Hamiltonian Monte Carlo (HMC) algorithms \cite{Hybrid_MCMC} seek to avoid quadratic slowdowns associated with diffusive ``random walk" behavior.  They do so by adding a notion of ``momentum" to the random walk, which encourages the underlying Markov chain to propose longer steps without incurring a large chance of rejection. For this reason HMC algorithms work especially well in high dimensions - concentration of the posterior measure $\pi$ causes most other MCMC algorithms to either propose steps that are very small or larger steps that are rejected with high probability (see \textit{e.g.,} \cite{roberts2001optimal,beskos2013optimal} for discussion of these heuristics).

In this section we review two commonly-used HMC algorithms.  Both algorithms are based on Hamiltonian dynamics, which we review here.  Fix a smooth function $H \, : \, \mathbb{R}^{2d} \mapsto \mathbb{R}^{+}$, called the \textit{Hamiltonian function}, and starting points $p_{0},q_{0} \in \mathbb{R}^{d}$. We then define \textit{Hamilton's equations} to be the pair of differential equations
\be \label{EqHamEq}
\frac{d}{dt} p(t) = - \frac{\partial H(p,q)}{\partial q}, \quad \frac{d}{dt} q(t) = \frac{\partial H(p,q)}{\partial p},
\ee
with initial condition $q(0) = q_{0}$, $p(0) = p_{0}$. Throughout the paper, we denote by $\gamma_{p,q}(t) = q(t)$ the first component of the solution to these equations with these initial conditions, so that \textit{e.g.}, $\gamma_{p,q}'(t) = p(t)$ for the ``standard" choice of Hamiltonian in Equation \eqref{EqStandardHam}.  The associated Hamiltonian function $H$ will always be clear from the context.

The solutions to Hamilton's equations have a number of special properties. Two of the most important are \textit{conservation of energy} and \textit{conservation of volume}:

\begin{enumerate}
\item \textbf{Conservation of Energy:} For $\gamma_{p,q}$ as above, we have 
\be \label{EqConsEn}
\frac{d}{dt} H(\gamma_{p,q}'(t), \gamma_{p,q}(t)) \equiv 0.
\ee
\item \textbf{Invariant distribution:} Assume that $\int_{p,q} e^{-H(p,q)} dq dp = 1$, and write $\mu_{H} = e^{-H(p,q)}$. If we sample $(P,Q) \sim \mu_{H}$, then we also have 
\be  \label{EqConsVol}
\mathcal{L}(\gamma_{P,Q}'(t), \gamma_{P,Q}(t)) = \mu_{H}
\ee for all $t \in \mathbb{R}^{+}$. This fact is a restatement of Liouville's theorem in probabilistic language.
\end{enumerate}

We refer the reader to \cite{lanczos1949variational} for proofs of these facts. The second fact motivates the standard choice of Hamiltonian
\be \label{EqStandardHam}
H(p,q) = -\log(\pi(q)) + \frac{1}{2} \| p \|^{2}.
\ee 
Under this measure, if $Q \sim \pi$ and $P \sim \mathcal{N}(0,1)$, then $\gamma_{P,Q}(t) \sim \pi$ for all $t \in \mathbb{R}$.
This choice leads to the isotropic-momentum HMC algorithm (Algorithm \ref{alg:isotropic_HMC}), developed in \cite{mehlig1992hybrid}:

\begin{algorithm}[H]
\caption{Isotropic-Momentum HMC (idealized symplectic integrator) \cite{mehlig1992hybrid}} \label{alg:isotropic_HMC}
\flushleft
\textbf{Input:} Starting point $q_0 \in \mathbb{R}^{d}$, integration time $T \in \mathbb{R}^{+}$, smooth target density $\pi: \mathbb{R}^n \rightarrow [0, \infty)$. \\
\textbf{Output:} Markov chain $q_1, q_2, \ldots$, with stationary distribution $\pi$. \\
\begin{algorithmic}[1]
\STATE Define $H(p,q) := -\mathrm{log}(\pi(q)) + \frac{1}{2} \| p \|^{2}$.
\FOR{$i = 1,2,\ldots$,}
\STATE Sample $p_i \sim \mathcal{N}(0,1)^d$. 
\STATE Set $q_{i+1} = \gamma_{p_{i}, q_{i}}(T)$.
\ENDFOR
\end{algorithmic}
\end{algorithm}

Riemannian Manifold HMC seeks to take longer steps by choosing initial momenta from a multivariate Gaussian distribution that agrees with the local geometry of the posterior density $\pi$. This is achieved by using trajectories that evolve according to Hamiltonian dynamics on a Riemmanian manifold, with metric defined by some positive definite matrix $G(q)$.  For the special case of isotropic-momentum HMC, we have $G(q) = I_d$.  Alternatively, one may choose $G(q)$ to be a regularization of the Hessian of $U(q)$, which acts as a local pre-conditioner for $U$ \cite{betancourt2013general}. This can result in much faster mixing for distributions that are not close to isotropic; for example, it avoids the bad performance of Algorithm \ref{alg:isotropic_HMC} on the nearly-degenerate Gaussian that is observed in Theorem \ref{ThmHmcDegenerate}. The algorithm is identical to Algorithm \ref{alg:isotropic_HMC}, except for the choice of the Hamiltonian $H$:

\begin{algorithm}[H]
\caption{Riemannian Manifold HMC (idealized symplectic integrator) \cite{Riemannian_HMC1, Riemannian_HMC2}} \label{alg:Riemannian_algorithm}
\flushleft
\textbf{Input:} Starting point $q_0 \in \mathbb{R}^{d}$, Integration time $T \in \mathbb{R}^{+}$, smooth target density $\pi: \mathbb{R}^d \rightarrow [0, \infty)$. \\
\textbf{Input:} Positive definite matrix $G(q)$ defining a Riemannian metric at every $q\in \mathbb{R}^d$. \\
\textbf{Output:} Markov chain $q_1, q_2, \ldots$, with stationary distribution $\pi$. \\
\begin{algorithmic}[1]
\STATE Define $H(p,q) := -\mathrm{log}(\pi(q)) +   \, \frac{1}{2}\log\left((2\pi)^d\det(G(q))\right) + \frac{1}{2} \langle p^\top G^{-1}(q), p \rangle$.%, where $c_d = \frac{1}{2}\mathrm{log}(\mathrm{\pi})^d$.
\FOR{$i = 1,2,\ldots$,}
\STATE Sample $p_i \sim \mathcal{N}(0, G^{-1}(q_i))$. 
\STATE Set $q_{i+1} = \gamma_{p_{i}, q_{i}}(T)$.
\ENDFOR
\end{algorithmic}
\end{algorithm}

We will usually make the following additional assumption on the Riemannian metric matrix $G(q)$ in this case:

\begin{assumption} \label{AssumptionBoundFIM}
Let $\lambda_{0}(q), \lambda_{1}(q)$ be the smallest and largest singular values of $G(q)$. Assume that for any compact set $A \subset \mathbb{R}^{d}$ we have  
\be 
\sup_{q \in A} \max \left( \frac{1}{\lambda_{0}(q)}, \, \lambda_{1}(q) \right) < \infty.
\ee 
\end{assumption}

Assumption \ref{AssumptionBoundFIM} says that the singular values of the 
metric $G$ are bounded away from $0$ and $\infty$ on compact sets. This is satisfied when $G$ is continuous and positive definite. More effort is required
when $G$ is allowed to be positive semi-definite; we do not study this case.

\subsection{Cheeger's inequality and the spectral gap}

We recall the basic definitions used to measure the efficiency of MCMC algorithms. Let $L$ be a reversible transition kernel with unique stationary distribution $\mu$ on $\mathbb{R}^{d}$. We view $L$ as an operator from $L_{2}(\pi)$ to itself:
\be
(L  f)(x) = \int_{y \in \mathbb{R}^{d}} L(x,dy) f(y).
\ee
The constant function is always an eigenfuncton of this operator, with eigenvalue 1. We define the space $W^{\perp} = \{ f \in L_{2}(\mu) \, : \, \int_{x} f(x) \mu(dx) = 0\}$ of functions that are orthogonal to the constant function, and denote by $L^{\perp}$ the restriction of the operator $L$ to the space $W^{\perp}$. We then define the \textit{spectral gap} $\rho$ of $L$ by the formula 
\be
\rho = \rho(L) \equiv 1 - \sup \{ |\lambda| \, : \, \lambda \in \mathrm{Spectrum}(L^{\perp}) \},
\ee
where $\mathrm{Spectrum}$ refers to the usual spectrum of an operator. If $L^{\perp}$ has a largest eigenvalue $|\lambda_2|$, then $\rho = 1-| \lambda_2 |$.  Geometric ergodicity of HMC algorithms was proved under very general conditions in \cite{livingstone2016geometric, Nawaf}, implying existence of a non-zero spectral gap under those conditions \cite{roberts1997geometric}.

Cheeger's inequality \cite{cheeger1970lower,lawler1988bounds} provides bounds for the spectral gap in terms of the ability of $L$ to move from any set to its complement in a single step. This ability is measured by the conductance $\Phi(L)$, which is defined by the pair of equations
\be 
\Phi(L) &= \inf_{S  \in \mathcal{A} \, : \, 0 < \mu(S) < \frac{1}{2}} \Phi(L,S) \\
\Phi(L,S) &= \frac{ \int_{x} \mathbbm{1}\{x \in S\} L(x,S^{c}) \mu(dx)}{\mu(S) },
\ee 
where $\mathcal{A}= \mathcal{A}(\mathbb{R}^{d})$ denote the usual collection of Lebesgue-measurable subsets of $\mathbb{R}^{d}$.

Cheeger's inequality for Markov chains, first proved in \cite{lawler1988bounds}, gives:
\begin{equation} \label{IneqCheegPoin}
\frac{\Phi(L)^2}{2} \leq \rho(L) \leq 2 \Phi(L).
\end{equation}

\section{Conductance of Hamiltonian Monte Carlo} \label{sec:HMC}

We derive equations for the conductance of HMC. Our main result, Theorem \ref{thm:1}, is based on the following two heuristics about HMC started at stationarity:
\begin{enumerate}
\item The conductance of a Markov chain measures the probability that a Markov chain will jump from some set $S$ to its complement in a single step. In the special case of HMC, this is roughly equivalent to the probability that the path $\gamma_{P,Q}$ will cross the boundary $\partial S$ an odd number of times.
\item By Liouville's theorem and reversibility, the \textit{rate} $\frac{d}{dt} | \gamma_{P,Q}([0,t]) \cap (\partial S)|$ at which the path $\gamma_{P,Q}$ intersects $\partial S$ is constant.
\end{enumerate}

Theorem \ref{thm:1} describes the conductance of a set purely in terms of this constant ``crossing rate" and the probability that the total number of crossings will be odd. Unfortunately, defining ``crossings" in a precise way requires a substantial amount of additional notation, which we give in Section \ref{SubsecNotationManifolds}. Readers primarily interested in upper bounds on the conductance may skip to Corollary \ref{cor:1}, which gives an easier-to-state upper bound on the conductance that is often close to sharp.

Although we introduce some restrictions on the set $S$ before giving our formula for the conductance of a \textit{particular set} $S$, we later prove that is possible to compute the conductance of \textit{the entire Markov chain} exactly using only the sets that satisfy these assumptions; see Lemma \ref{LemmSuffLoc} for details. Thus, our formula can in fact be used to compute the conductance exactly.

\subsection{Notation Related to Manifolds and Intersections} \label{SubsecNotationManifolds}

Fix a set $S \subset \mathbb{R}^{d}$ whose topological boundary $\partial S$  is an embedded $(d-1)$-dimensional manifold. Fix also an integration time $T \in \mathbb{R}^{+}$. In this special case, we define the \textit{number of intersections} $N_{\partial S}$ as follows:

\begin{defn} \label{DefNumInt}
For fixed $T > 0$, define the set $\mathcal{G}_{S} = \mathcal{G}_{S}(T) \subset \mathbb{R}^{2d}$ to be all pairs $(p,q) \in \mathbb{R}^{2d}$ satisfying:
\begin{enumerate}
\item The function $\gamma_{p,q} \, : \, \mathbb{R} \mapsto \mathbb{R}^{d}$ 
is \textit{transverse} to the manifold $\partial S$.
\item The set $\{ t \in [0,T] \, : \, \gamma_{p,q}(t) \in \partial S\}$ is finite.
\end{enumerate}

For $(p,q) \in \mathcal{G}_{S}$, define
\be
N_{\partial S} = N_{\partial S}(p,q) \equiv |\{ t \in [0,T] \, : \, \gamma_{p,q}(t) \in \partial S\}|.
\ee
\end{defn}

To recall a precise definition of a \textit{transverse} intersection, see \textit{e.g.,} \cite{guillemin2010differential}. In this paper, we are primarily interested in the following property of transverse intersections. Let $\partial S \subset \mathbb{R}^{d}$ be a closed surface that partitions $\mathbb{R}^{d}$ into two open pieces $P_{1}'$, $P_{2}'$ with closures $P_{1},P_{2}$, and let $\gamma \, : \, [0,1] \mapsto \mathbb{R}^{d}$ be a curve that is transverse to $\partial S$. Assume that $|\gamma([0,1]) \cap \partial S| < \infty$. Then $\gamma(0) \notin \partial S, \, \gamma(1)$ are in the same closed piece of $\mathbb{R}^{d}$ if and only if $|\gamma([0,1]) \cap \partial S| < \infty$ is even.

In other words: if all intersections are transverse, we can tell which side of a curve a path will end on by simply counting the number of intersections. This is not true for non-transverse intersections. For example, consider the intersection of the circle given by $x^{2} + y^{2} = 1$ and the curve $\gamma(t)=(2t - 1, 1)$; they intersect at exactly one point but the curve never enters the inside of the circle. See also Figure \ref{FigTransverse}.

\begin{figure}[H]
\includegraphics[scale=0.4]{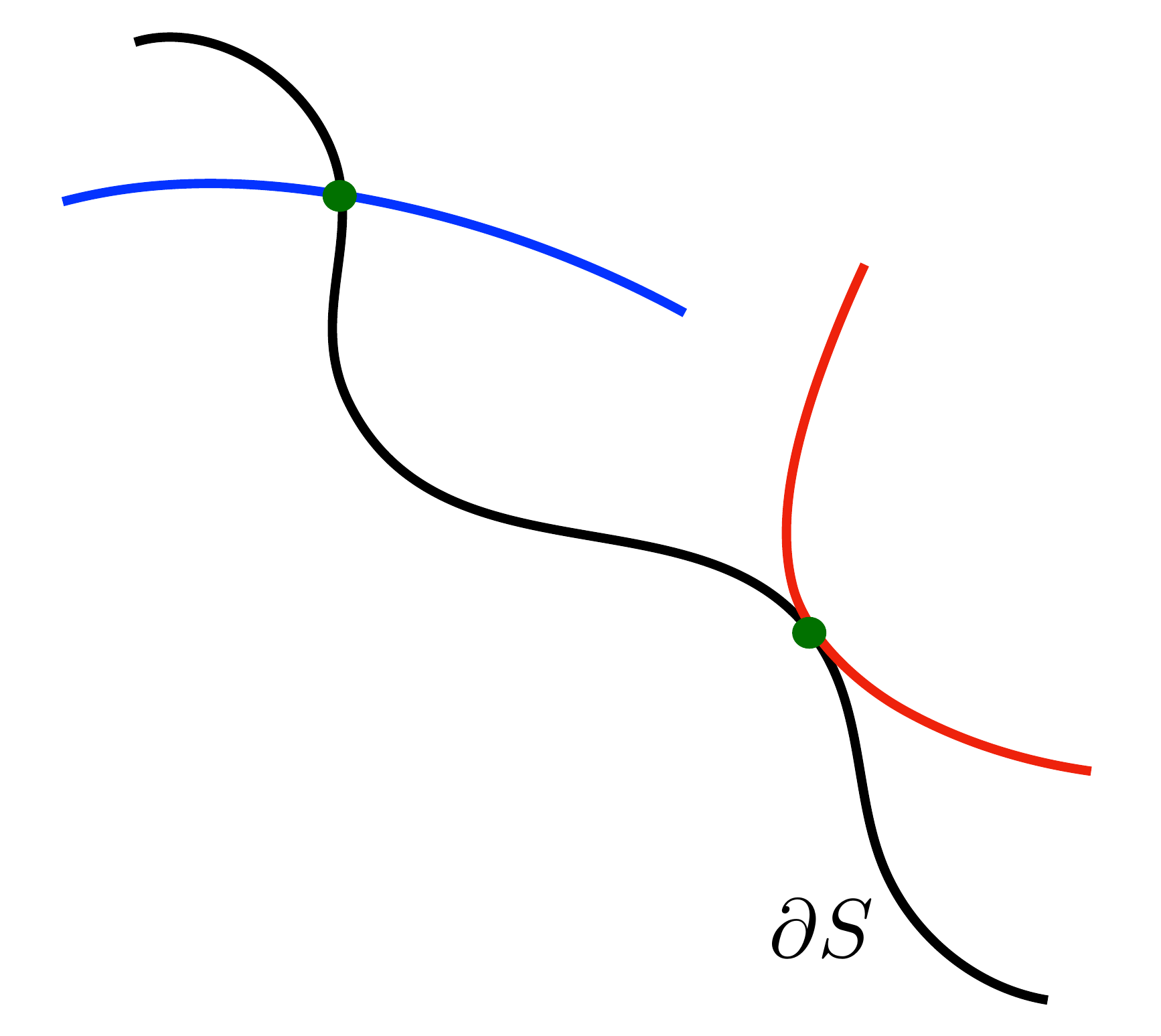}
\caption{A boundary $\partial S$ is shown in black. The intersection of $\partial S$ with the blue curve is transverse; the intersection with the red curve is not.\label{FigTransverse}}
\end{figure}

We next show that this definition of $N_{\partial S}(p,q)$ can easily be extended to all $(p,q) \in \mathbb{R}^{2d}$, under modest assumptions about $S$.  We need the following simple notation:

\begin{itemize}
\item For $q \in \partial S$, define $\eta(q)$ to be the unit normal vector of $\partial S$ at $q$ that points away from the interior of $S$, and for $p \in \mathbb{R}^{d}$ define
\be
p_{q} = \langle G^{-1}(q) p, \eta(q) \rangle
\ee

to be the component of $p$ in the direction orthogonal to $\partial S$ at q.
\item For $q \in \partial S$, define the set 

\be 
\mathcal{P}^{+}_{S}(q) = \{ p \in \mathbb{R}^{d} \, : \, p_{q} > 0 \}
\ee

to be the half-space of momentum vectors pointing away from $S$ at $q\in \partial S$.
\end{itemize}

For $x \in \partial S$, denote by $\mathcal{T}_{x}$ the tangent space of $\partial S$ at $x$. We view this tangent space as being embedded in the same copy of $\mathbb{R}^{d}$ as $S$ and being based at the point $x$ (so that, for example, it will generally not include 0). Denote by $\mathbbm{Proj}_{x} \, : \, \mathbb{R}^{d} \mapsto \mathcal{T}_{x}$ the projection map from $\mathbb{R}^{d}$ to the tangent space. We assume:

\begin{assumption} [Locally Well-Behaved Manifold] \label{DefLocallyWellBehaved}
We say that $S$ is \textit{locally well-behaved} if, for every compact set $A \subset \mathbb{R}^{d}$, there exist constants $0 < \epsilon, C < \infty$ so that:
\begin{enumerate}
\item For every ball $A' \subset A$ of diameter less than $\epsilon$, the set $A' \cap \partial S$ has a single component, and
\item For every $x, y \in (A \cap \partial S)$,
\be
\| y - \mathbbm{Proj}_{x}(y) \| \leq C \|x - \mathbbm{Proj}_{x}(y) \|^{2}
\ee
and 
\be
\| \eta(x) - \eta(y) \| \leq C \|x-y \|.
\ee
\end{enumerate}
\end{assumption}

\begin{remark}
We have not tried to provide the weakest-possible assumptions in Assumption \ref{DefLocallyWellBehaved}. Since Lemma \ref{LemmSuffLoc} implies that the conductance of most realistic HMC chains can be computed entirely in terms of sets that satisfy Assumption \ref{DefLocallyWellBehaved}, these assumptions are weak enough for our purposes. Here, we explain why something like Assumption \ref{DefLocallyWellBehaved} is needed at all.

The goal is to rule out the possibility that solutions to Hamilton's equations will pass through $\partial S$ very many times over very short intervals. To give a simple pathological example, we wish to avoid sets such as
\be  \label{BadSetSDef}
S_{\infty} = \cup_{n \in \mathbb{N}} [0,1] \times [(4n)^{-2}, (4n+1)^{-2}].
\ee

The problem with $S_{\infty}$ is that the boundary $\partial S$  contains a countably infinite union of parallel lines within the compact set $[0,1]^{2}$, and so arbitrarily short Hamiltonian paths can cross $\partial S$ arbitrarily (or even infinitely) often. 
\end{remark}

We next note that $N_{\partial S}$ can be defined $\mathbb{R}^{2d}$-almost everywhere:

\begin{lemma} \label{LemCountingPossible}
Set notation as above. Then the set $\mathcal{G}_{S}^{c}(T)$ has measure 0 for any $T \in \mathbb{R}^{+}$.
\end{lemma}

\begin{proof}
The proof is deferred to Appendix \ref{SecCountingPossible}.
\end{proof}

By this lemma, for all fixed $T \in \mathbb{R}^{+}$ we can define a measurable function $N_{\partial S} \, : \, \mathbb{R}^{2d} \mapsto \mathbb{N}$ that agrees with Definition \ref{DefNumInt} for almost every value of $(p,q) \in \mathbb{R}^{2d}$. Throughout the rest of the paper, we will always assume that all intersections are transverse. By this lemma, the complement of the set $\cap_{n \in \mathbb{N}} \mathcal{G}_{S}(n)$ has measure 0, and so this assumption will not influence the results of any calculations.

Next, we define two families of tilted measures on $\mathbb{R}^{2d}$:

\begin{defn} [Tilted Measures]

Let $\tilde{\mathbb{P}}$ be the probability measure on $\mathbb{R}^{2d}$ with density
\be
\tilde{\mathbb{P}}(p,q) \propto e^{-H(p,q)} \, \mathbbm{1} \{ N_{\partial S}(p,q) \geq 1 \}.
\ee

Define $\mathbb{Q}$ to be the probability measure on $\mathbb{R}^{2d}$ with density
\be 
\mathbb{Q}(p,q) = \tilde{\mathbb{P}}(p,q) \cdot N_{\partial S}(p,q).
\ee
\end{defn}

For a set $A \subset \mathbb{R}^{d}$ and a constant $c > 0$, define the $c$-thickening of $A$ to be 
\be 
A_{c} = \{x \in \mathbb{R}^{d} \, : \, \inf_{a \in A} \, \|x-a\| \leq c\}.
\ee

\subsection{Main Conductance Formula}

Our main theorem is:

\begin{thm} \label{thm:1} 

Let $T \in \mathbb{R}^{+}$, let $\pi(q)$ be any smooth probability density on $\mathbb{R}^d$, let $K$ and $H$ be the transition kernel and Hamiltonian of either Algorithm \ref{alg:isotropic_HMC} or Algorithm \ref{alg:Riemannian_algorithm} with these parameters, and let $\mu_{H}(p,q) \propto e^{-H(p,q)}$ be the associated Hamiltonian measure. Let $S \subset \mathbb{R}^{d}$ be any subset whose boundary $\partial S$ is a smooth manifold satisfying Assumption \ref{DefLocallyWellBehaved}. Finally, in the case of  Algorithm \ref{alg:Riemannian_algorithm}, let Assumption \ref{AssumptionBoundFIM} also hold.

Then the conductance of $K$ satisfies
\be\label{eq:HMC_traditional_Cheeger}
\Phi(K,S) = \Phi^+ \cdot \mathbb{E}\bigg[\frac{1}{N_{\partial S}(P,Q)} \cdot \mathbbm{1}\{N_{\partial S}(P,Q) \, \mathrm{odd}\}\bigg] \bigg/ (\pi(S)),
\ee

where the expectation is taken with respect to the random variables $(P,Q) \sim \mathbb{Q}$ and the total positive flux $\Phi^+$ is given by
\be  \label{EqPosFluxConc1}
\Phi^+ = \frac{1}{2} T \cdot \int_{\partial S} \int_{\mathbb{R}^d} \mu_{H}(p,q) \cdot    |p_q| \mathrm{d}p\mathrm{d}q.
\ee

In the case of Algorithm \ref{alg:isotropic_HMC}, this formula for $\Phi^+$ reduces to the simpler expression

\be  \label{EqPosFluxConc2}
\Phi^+ = \frac{1}{2}T \cdot \int_{\partial S} \pi(q) \mathrm{d}q.
\ee 

\end{thm}

\begin{proof}
The proof of this result is given in Appendix \ref{SecPfThm1}.
\end{proof}

Note that, in Equations \eqref{EqPosFluxConc1} and \eqref{EqPosFluxConc2}, the integral over $\partial S$ is taken with respect to the volume measure on the $(d-1)$-dimensional manifold $\partial S$, not with respect to the Lebesgue measure on the $d$-dimensional space $\mathbb{R}^{d}$.

In many applications, it is more important to get a good bound on the conductance than to compute it directly. Observing that $\mathbb{E}_\mathbb{Q}\bigg[\frac{1}{N_{\partial S}} \cdot \mathbbm{1}\{N_{\partial S} \, \mathrm{odd}\}\bigg] \leq 1$, we have the following useful consequence of Theorem \ref{thm:1}:

\begin{cor}[Simple Conductance Bound] \label{cor:1} 

Set notation as in Theorem \ref{thm:1}. Then the conductance for the Isotropic-Momentum HMC algorithm is bounded by 

\be 
\Phi(K,S) \leq \frac{1}{2}T\frac{\int_{\partial S} \pi(q) \mathrm{d}q \cdot }{\pi(S)}. 
\ee 
\end{cor}

An immediate consequence of this bound is that the time it takes energy-conserving Hamiltonian Markov chains to search for sub-Gaussian modes grows exponentially with both the dimension and the distance between modes. Numerical simulations for various two-mode densities that approximate Gaussian mixture models (Figure \ref{figure:simulation}) suggest that this upper bound is nearly tight in many cases where $T$ is not too large.

\subsection{Removing Assumption \ref{DefLocallyWellBehaved}}

Theorem \ref{thm:1} only applies to sets that satisfy Assumption \ref{DefLocallyWellBehaved}. Fortunately, it is not necessary to consider any other sets to compute the conductance of typical HMC Markov chains:

\begin{lemma} [Sufficiency of Compact Locally-Good Manifolds] \label{LemmSuffLoc}
Let $\pi$ be a distribution with twice-differentiable density, let $T > 0$, and let $K_{T}$ be the transition kernel given by Algorithm \ref{alg:isotropic_HMC} with these parameters. Let $\mathfrak{B}$ be the Lebesgue measurable subsets of $\mathbb{R}^{d}$, and let $\mathfrak{A} \subset \mathfrak{B}$ be the subsets that also satisfy Assumption \ref{DefLocallyWellBehaved}. Then
\be \label{IneqSuffLoc1}
\inf_{S \in \mathfrak{A}} \Phi(K_{T},S) = \inf_{S \in \mathfrak{B}} \Phi(K_{T},S).
\ee 
\end{lemma}

\begin{proof}
The proof is deferred to Appendix \ref{SecLemmSuffLoc}.
\end{proof}

\section{Examples and Applications} \label{SecAppl}

We illustrate our conductance bounds with two examples:

\begin{enumerate}
\item In Section \ref{SubsecAppMix}, we use Theorem \ref{thm:1} to compute the Cheeger constant of an HMC algorithm targetting a mixture of Gaussians. We also compute the spectral gap of this Markov chain, showing that it is roughly equal to the Cheeger constant. In particular, the \textit{upper} bound of Inequality \eqref{IneqCheegPoin} is close to sharp in this example.

\item In Section \ref{SubsecAppDeg}, we use Theorem \ref{thm:1} to compute the Cheeger constant of an HMC algorithm targetting a multivariate Gaussian with nearly-singular covariance matrix. We also compute the spectral gap of this Markov chain, showing that it is roughly equal to the \textit{square} of the Cheeger constant.  In particular, the \textit{lower} bound of Inequality \eqref{IneqCheegPoin} is close to sharp.
\item In Section \ref{SubsecAppNum}, we numerically compute the spectral gap and Cheeger constant of an HMC algorithm to illustrate Theorem \ref{thm:1}.
\end{enumerate}

\subsection{Application: HMC Targetting Mixture of Gaussians} \label{SubsecAppMix}

For $\sigma > 0$, define the mixture distribution 
\be  \label{EqDefMixDistBasic}
\pi_{\sigma} = \frac{1}{2} \mathcal{N}(-1,\sigma^2) + \frac{1}{2} \mathcal{N}(1,\sigma^2)
\ee  
and denote its density by $f_{\sigma}$. Let $K_{\sigma}$ be the transition kernel of Algorithm \ref{alg:isotropic_HMC} with target distribution $\pi = \pi_{\sigma}$ and time-step $T = T_{\sigma} \equiv \sigma$. Denote by $\lambda_{\sigma}$ the relaxation time of $K_{\sigma}$, and denote by $\Phi_{\sigma} = \Phi(K_{\sigma}, (-\infty,0))$ the Cheeger constant associated with kernel $K_{\sigma}$ and set $(-\infty,0)$.

For fixed $x \in (-\infty,0)$, let $\{X_{t}^{(\sigma)}\}_{t \in \mathbb{N}}$ be a Markov chain with transition kernel $K_{\sigma}$ and initial point $X_{1}^{(\sigma)} = x$. Define the hitting time 
\be \label{EqDefTauSigmaX}
\tau_{x}^{(\sigma)} = \inf \{ t > 0 \, : \, X_{t}^{(\sigma)} \notin (-\infty,0)\}.
\ee

We show:

\begin{thm} [HMC for Multimodal Distributions] \label{ThmHmcMultimodal}
The Cheeger constant $\Phi_{\sigma}$ satisfies 
\be  \label{EqMultiCheegAsym}
\lim_{\sigma \rightarrow 0} (-2 \sigma^{2}) \, \log(\Phi_{\sigma}) = 1.
\ee
Furthermore, for all $\epsilon > 0$ and fixed $x \in (-\infty,0)$, the hitting time $\tau_{x}^{(\sigma)}$ satisfies 
\be \label{EqMultiCheegHitting}
\lim_{\sigma \rightarrow 0} \P[\frac{\log(\tau_{x}^{(\sigma)})}{\log(\Phi_{\sigma})} < 1 + \epsilon] = 1
\ee 
and the relaxation time satisfies 
\be   \label{IneqRelMulti}
\lim_{\sigma \rightarrow 0} \frac{\log(\lambda_{\sigma})}{\log(\Phi_{\sigma})} = \lim_{\sigma \rightarrow 0} \frac{\log(\lambda_{\sigma})}{\log(\Phi(K_{\sigma}))} = 1.
\ee 
\end{thm}

We defer the proof to Appendix \ref{AppSubsecPfHmcMulti}. Note that this result exactly matches the spectral gap for optimally-tuned random-walk Metropolis algorithm with the same target distribution (see Theorem 3 of the our companion paper \cite{mangoubi2018simple}).

We also observe that this result implies the Cheeger constant $\Phi(K_{\sigma})$ of $K_{\sigma}$ is close to the bottleneck ratio $\Phi_{\sigma} = \Phi(K_{\sigma}, (-\infty,0))$ associated with the set $(-\infty,0)$, at least for $\sigma$ very small. The set $(-\infty,0)$ is of course a natural guess for the set with the ``worst" conductance, though we do not know of any simple argument that would prove something like this. In fact, deriving simple criteria to verify this fact was the motivation behind our companion paper \cite{mangoubi2018simple}.

\subsection{High Dimensional Analog} \label{sec:highdim} Almost all of the work in the proof of Theorem \ref{ThmHmcMultimodal} is checking that, in fact, the spectral gap is not much \textit{smaller} than the natural guess and upper bound $\Phi(K_{\sigma}, (-\infty,0))$; computing a sharp upper bound $\Phi(K_{\sigma}, (-\infty,0))$ itself is straightforward using Theorem \ref{thm:1}. This remains true in higher-dimensional examples. For example, considering the mixture distribution
\be 
\pi_{\sigma} = \frac{1}{2} \mathcal{N}((-1,0,\ldots,0),\sigma^2 \mathrm{Id}) + \frac{1}{2}\mathcal{N}((1,0,\ldots,0),\sigma^2 \mathrm{Id})
\ee  
on $\mathbb{R}^{d}$. It is natural to guess that, for the HMC algorithm with typical time step $T_{\sigma} = O(\sigma)$, the set $\{ x \in \mathbb{R}^{d} \, : \, x[1] < 0 \}$ has (nearly) the worst conductance. Theorem \ref{thm:1} can be used to get a very good estimate of this conductance, using the same calculation as the start of Theorem \ref{ThmHmcMultimodal}.

\subsection{Choice of Integration Time and Computational Cost}\label{sec:RemChoiceInt}
It is natural to ask: why do we study the choice $T_{\sigma} = \sigma$, rather than some other choice of integration time? The simplest answer is that it is impossible to substantially improve the performance of the algorithm by choosing a larger integration time. To make this statement precise, we note that the computational cost of running a single step of the HMC algorithm is approximately proportional to the integration time $T$. Thus, the effective computational cost of a sample from a target distribution $\pi$ is roughly the ratio of the integration time to the spectral gap (this is a standard way to compare the efficiency of MCMC algorithms with widely differing costs per step - see  \cite{bornn2017use,sherlock2017pseudo} and references therein). It is this effective computational cost that is bounded below: we always have
\be 
\limsup_{\sigma \rightarrow 0} 2 \sigma^{2} (\log(\lambda_{\sigma}) - \log(T_{\sigma})) \leq 1,
\ee for all $T_{\sigma} \geq \sigma$ (see discussion in Section \ref{SecHmcImp} for a proof of this fact). Thus, $T_{\sigma} = \sigma$ is essentially the optimal choice for $T_{\sigma}$ in this case.

Having said this, we find this simple answer slightly misleading. In practice, one never knows the optimal value of $T_{\sigma}$. Instead, one often tunes an MCMC algorithm according to some ``Goldilocks principle."\footnote{The first use of Goldilocks analogy in this context is by Jeff Rosenthal.} For HMC, one chooses $T_{\sigma}$  so that the probability of an HMC trajectory making a ``U-turn" (in the sense of \cite{NUTS}) is not too close to 0 and not too close to 1. In this example, this means choosig $T_{\sigma} = \Theta(\sigma)$. Note that this heuristic is very similar to the ``Goldilocks principle" for tuning RWM: one should choose the standard deviation of the proposal distribution so that the probability of rejecting a proposal is not too close to 0 and not too close to 1. This popular tuning choice is exactly the one that we study in our companion paper \cite{mangoubi2018simple}, even though it is \textit{not} close to optimal in that context.

\subsection{Application: HMC Targetting Degenerate Multivariate Gaussian} \label{SubsecAppDeg}

This section is motivated by the study of the standard HMC Algorithm \ref{alg:isotropic_HMC} with target distribution of the form $\mathcal{N}(0, M_{\sigma})$, where the 2-dimensional covariance matrix $M_{\sigma}$ is given by:  

\be 
   M_{\sigma}=
  \left[ {\begin{array}{cc}
   1 & 0 \\
   0 & \sigma^{2} \\
  \end{array} } \right].
\ee

We next observe that the target distribution $\mathcal{N}(0,M_{\sigma})$ is special: if $\{X_{t}\}_{t \in \mathbb{N}}$ is a Markov chain drawn from Algorithm \ref{alg:isotropic_HMC} targetting $\mathcal{N}(0,M_{\sigma})$, then the coordinate sequences $\{X_{t}[1]\}_{t \in \mathbb{N}}$ and $\{X_{t}[2]\}_{t \in \mathbb{N}}$ are each Markov chains as well. Furthermore, they both evolve independently, and they evolve according to Algorithm \ref{alg:isotropic_HMC}, with $\{X_{t}[1]\}_{t \in \mathbb{N}}$ targetting $\mathcal{N}(0,1)$ and $\{X_{t}[2]\}_{t \in \mathbb{N}}$ targetting $\mathcal{N}(0,\sigma^{2})$. Thus, rather than analyzing the full chain $\{X_{t}\}_{t \in \mathbb{N}}$, to calculate the spectral gap it is enough to analyze the slower-mixing marginal chain $\{X_{t}[1]\}_{t \in \mathbb{N}}$.

Denote by $K_{\sigma}$ the transition kernel associated with target distribution $\mathcal{N}(0,1)$ and integration time $T_{\sigma}$ satisfying $T_{\sigma} = o(1)$. Denote by $\Phi_{\sigma}$ the Cheeger constant associated with this transition kernel and the set $(-\infty,0)$; denote by $\rho_{\sigma}$ the spectral gap of $K_{\sigma}$. We have:

\begin{thm} \label{ThmHmcDegenerate}
The Cheeger constant $\Phi_{\sigma}$ satisfies 
\be  \label{EqMultiCheegDeg}
\lim_{\sigma \rightarrow 0} \frac{\log(\Phi_{\sigma})}{\log(T_{\sigma})} \leq 1.
\ee 
Furthermore, the relaxation time satisfies 
\be  \label{IneqRelDeg}
\limsup_{\sigma \rightarrow 0} \frac{\log(\rho_{\sigma})}{\log(T_{\sigma})} \geq \frac{1}{2}.
\ee
\end{thm}

The proof is deferred to the appendix. Note that we do not give a bound for the hitting time in this example. This is not an accident: the target distribution is unimodal, and so the HMC algorithm does not exhibit metastability and the fluctuations of the hitting time remain large relative to the relaxation time as $\sigma$ goes to 0. 

\begin{remark}
This example is studied in Example 2.1 of \cite{rabee2018HMCcoup}. We observe that our results are sharper in a few ways (we allow larger integration times; we obtain the correct order of the dependence of the spectral gap on $\sigma$), but are also the results of exact computations rather than general theorems.

We mention that our own previous work \cite{mangoubi2017concave} and that of \cite{holmes2014curvature} gives similar (non-sharp) estimates to \cite{rabee2018HMCcoup} in this example. We think that finding general bounds that give sharp answers in this prototypical case to be an interesting problem.
\end{remark}

\subsection{Numerical Simulation: Spectral Gap of Hamiltonian Monte Carlo} \label{SubsecAppNum}

In this section we plot the spectral gap (Figure \ref{figure:simulation}) of an Isotropic-Momentum HMC algorithm sampling a two-mode density; the plot was generated by numerically diagonalizing an analytical solution for the transition matrix of the HMC Markov chain.  The results of the calculation agree closely with the upper bounds on the spectral gap given by applying Corollary \ref{cor:1} with Inequality \eqref{IneqCheegPoin}.

In this simulation we computed the spectral gap for the Isotropic-Momentum HMC algorithm with the stationary distributions of the form $\pi_{a}(q) = \frac{1}{2F_{\mathcal{N}(0,1)}(a)}\mathrm{max}(f_{\mathcal{N}(0,1)}(q-a),f_{\mathcal{N}(0,1)}(q+a))$, where $f_{\mathcal{N}(0,1)}$ and $F_{\mathcal{N}(0,1)}$ denote the CDF and PDF of the standard normal density.  Note that $\pi_{a}(q)$ approximates the Gaussian mixture model $\tilde{\pi}_{a}(q) =  \frac{1}{2}f_{\mathcal{N}(0,1)}(q-a) + \frac{1}{2}f_{\mathcal{N}(0,1)}(q+a)$; indeed $\lim_{a \rightarrow \infty} \| \pi_{a} - \tilde{\pi}_{a} \|_{\mathrm{TV}} = 0$.

As suggested by the formula for the conductance in Theorem \ref{thm:1}, the spectral gap is bounded above by a linear function of $T$, and in fact increases approximately linearly with $T$ when $a=0$ for $T\leq \frac{\pi}{2}$. Note that, for fixed $a$ small, the spectral gap $(1 - \lambda_{2})$ looks like a periodic function of $T$. This is due to the fact that the trajectories themselves are very close to periodic with period $\geq\frac{\pi}{2}$, meaning that the expectation term in Equation \eqref{eq:HMC_traditional_Cheeger} also varies (approximately) periodically with T.  The exponential decay in $a^2$ is explained by considering the set $S = (-\infty,0)$ and noting that $\int_{\partial S} \pi_{a}(q) \mathrm{d}q = f_{\mathcal{N}(0,1)}(q-a)$, so the corresponding term in Corollary \ref{cor:1} of Theorem \ref{thm:1}  decreases exponentially in $a^2$.

\begin{figure}[H]
\includegraphics[scale=0.3]{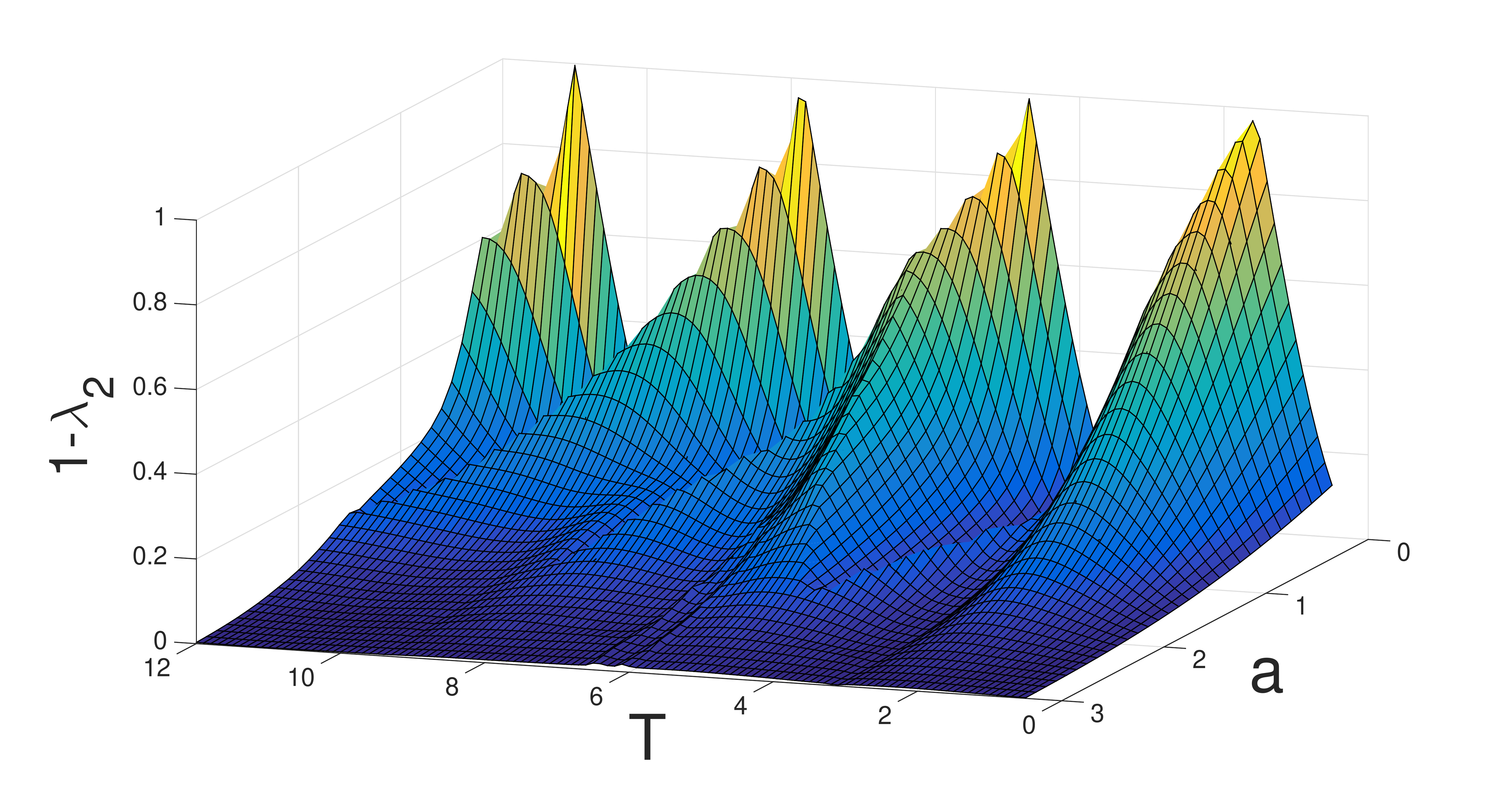}
\caption{The spectral gap for the Isotropic-Momentum HMC algorithm with stationary distribution $\pi(q) = \frac{1}{2F_{\mathcal{N}(0,1)}(a)}\mathrm{max}(f_{\mathcal{N}(0,1)}(q-a),f_{\mathcal{N}(0,1)}(q-a))$, for different inter-modal distances $2a$  and different Hamiltonian trajectory times $T$.  The results agree closely with the bound in Theorem \ref{thm:1}.\label{figure:simulation}}
\end{figure}

\section{Discussion and Open Problems} \label{SecCompMulti} 

We give some consequences of Theorem \ref{thm:1}, and discuss open questions related to the performance of HMC.

\subsection{Bounds on HMC Improvements} \label{SecHmcImp}

One immediate consequence of Theorem \ref{thm:1} and Lemma \ref{LemmSuffLoc} is that it is not possible to dramatically improve the performance of Algorithm \ref{alg:isotropic_HMC} by tuning the trajectory integration time $T$. We give a quick discussion of this fact and some related open problems.

Fix a probability distribution  with twice-differentiable density $\pi$ and denote by $K_{T}$ the transition kernel defined by Algorithm \ref{alg:isotropic_HMC} with parameters $\pi$ and $T > 0$. Define 
\be 
\Phi_{0}(\pi,S) &= \frac{1}{2} \frac{\int_{\partial S} \pi(q) \mathrm{d}q  }{\pi(S)}\\
\Phi_{0}(\pi) &= \inf_{S \in \mathfrak{A}}  \Phi_{0}(\pi,S), \\
\ee 
for $S \in \mathfrak{A}$. By Corollary \ref{cor:1} of Theorem \ref{thm:1}, combined with Lemma \ref{LemmSuffLoc}, we have the upper bound
\be \label{IneqLinIncCond}
T^{-1} \Phi(K_{T}) \leq  \Phi_{0}(\pi),
\ee 
so that $\Phi(K_{T})$ is bounded by a linear function in the integration time $T$. 

This observation is already in stark contrast to RWM as given in Algorithm \ref{alg:RWM}, whose tuning parameter $\epsilon > 0$ can have very large impacts on performance. As an illustration of this large change in performance, consider the family of transition kernels $\{ Q_{\sigma} \}_{0<\sigma <1}$ given by Algorithm \ref{alg:RWM} with target distribution $\pi_{\sigma}$ given in Equation \eqref{EqDefMixDistBasic} and tuning parameter $\epsilon \equiv 10$ for all $\sigma$. It is a straightforward exercise \footnote{ For example, one could apply Theorem 5 of \cite{rosenthal1995minorization} with ``small set" $(-5,5)$ and ``Lyapunov function" $V(x) = e^{\|x\|}$ to obtain a bound on the rate of convergence of the walk to stationarity. One could then apply Theorem 2.1 of \cite{roberts1997geometric} to convert this bound on the convergence rate to a bound on the spectral gap.} to check that $\log(\rho(Q_{\sigma})^{-1}) = O(\log(\sigma^{-1}))$ for $\sigma$ small. Comparing this with the result of Theorem 3 of \cite{mangoubi2018simple}, we see that changing the tuning parameter $\epsilon$ of Algorithm \ref{alg:RWM} by a factor on the order of $\sigma$ will vastly decrease the size of the log of the relaxation time, from $\Theta(\sigma^{-2})$  to $O(-\log(\sigma))$. 

This dramatic performance improvement does not have obvious practical applications, since choosing a good tuning parameter for multimodal targets requires a more detailed understanding of the mode than a user will typically have. Our point is just that such dramatic improvements are possible for RWM with multimodal targets, while they are not possible for HMC.

We interpret the linear bound \eqref{IneqLinIncCond} in terms of computational cost. Roughly speaking, the computational cost of a step of Algorithm \ref{alg:isotropic_HMC} is  proportional to the integration time $T$. Paraphrasing Inequality \eqref{IneqLinIncCond}, it is impossible to improve the cost-normalized conductance of Algorithm \ref{alg:isotropic_HMC} by increasing the integration time $T$. This immediately implies, via Cheeger's inequality \eqref{IneqCheegPoin}, that the  spectral gap is bounded by a quadratic in $T$. We point out that this behaviour is very similar to that of ``lifted" Markov chains (see \cite{chen1999lifting} for a definition). Roughly speaking, like HMC, ``lifted" Markov chains attempt to combat diffusive behaviour of RWM by adding an abstract notion of ``momentum." The central idea behind the non-improvement theorems for lifted chains in \cite{chen1999lifting,ramanan2017bounds} is that it is impossible to increase the conductance of a Markov chain by ``lifting" it. By Cheeger's inequality, this gives an upper bound on the best-possible improvement due to lifting.

We close this discussion with an open question. We have shown that the conductance \textit{of a given set} can increase at most \textit{linearly} in $T$. Via Cheeger's inequality \eqref{IneqCheegPoin}, this suggests that the spectral gap should satisfy a similar \textit{quadratic} inequality in $T$. To make this rigorous, it would be sufficient to show:

\textbf{Open Problem 1:} Prove the limit
\be \label{EqLimOpenQuestion}
\lim_{T \rightarrow 0} \inf_{S \in \mathfrak{A}} T^{-1} \Phi(K_{T},S) = \inf_{S \in \mathfrak{A}} \Phi_{0}(\pi,S). 
\ee

Note that Equality \eqref{IneqE4Small} already implies that, for fixed $S \in \mathfrak{A}$, 
\be 
\lim_{T \rightarrow 0}  T^{-1} \Phi(K_{T},S) =  \Phi_{0}(\pi,S). 
\ee

\subsection{Multimodal Targets}

In Theorem \ref{ThmHmcMultimodal} and Theorem 3 of the companion paper \cite{mangoubi2018simple}, we showed that Algorithm \ref{alg:isotropic_HMC} has very similar performance to Algorithm \ref{alg:RWM} for a strongly multimodal example. However, our comparison is not direct: we prove that the two algorithms have similar spectral gaps by laboriously computing the spectral gaps of both algorithms. We suspect that this behaviour is quite general, and propose the following informal problem:

\textbf{Open Question 2:} Let $\pi_{\sigma}$ be a mixture distribution of the form 
\be 
\pi_{\sigma}(x) \propto \sum_{i=1}^{k} \mu_{i} f_{i}(\frac{x}{\sigma}),
\ee 
where $f_{1},\ldots,f_{k}$ are the densities of probability distributions and $\mu_{1},\ldots,\mu_{k} \geq 0$ are fixed weights that sum to $\sum_{i=1}^{k} \mu_{i} = 1$, $a_{1},\ldots,a_{k}$.

Let $K_{\sigma}$, $Q_{\sigma}$ be the transition kernels of Algorithms \ref{alg:isotropic_HMC} and \ref{alg:RWM} respectively, with target distribution $\pi_{\sigma}$, integration time $T_{\sigma} \propto \sigma$ and standard deviation $\epsilon_{\sigma} \propto \sigma$. We wish to know: what are sufficient conditions for $\{ \mu_{i}\}_{i=1}^{k}$ so that 
\be \label{LimConj2Strong}
\lim_{\sigma \rightarrow 0} \frac{\log(\rho(K_{\sigma}))}{\log(\rho(Q_{\sigma}))} \geq 1?
\ee 

Note that the constant ``1" in the limit \eqref{LimConj2Strong} is important. We  expect all algorithms to perform poorly for multimodal targets, and thus for the limit \eqref{LimConj2Strong} to be strictly greater than 0. A limit of ``1" would suggest that HMC exhibits \textit{no} asymptotic improvement over RWM; a limit of \textit{e.g.} $\frac{1}{2}$ would suggest a quadratic improvement over RWM, which is substantial.

We conjecture that Inequality \eqref{LimConj2Strong} holds as long as the level sets of $f_{i}$ are fairly ball-like; for example, if the sets $S_{i,C} \equiv \{x \, : \, f_{i}(x) \leq C\}$ are all convex, with the ratio of the inner and outer radii
\be 
r_{\mathrm{inner}}(i,C) &\equiv \sup\{r \, : \, \exists \, x \in S_{i,C} \text{ s.t. } B_{r}(x) \subset S_{i,C} \} \\
r_{\mathrm{outer}}(i,C) &\equiv \inf\{r \, : \, \exists \, x \in S_{i,C} \text{ s.t. } B_{r}(x) \supset S_{i,C} \} \\
\ee 
uniformly bounded.  We mention that we do not expect this to hold in complete generality. For example, if the level sets $S_{i,C}$ of $\pi_{\sigma}$ for $0 <C \ll e^{\sigma^{-1}}$ look like Figure \ref{FigBanana},

%\begin{figure}[H]
%\includegraphics[trim=2cm 2.2cm 0cm 1.5cm, width=12.6cm]{banana}
%\caption{We see a small deep mode next to a long and skinny one. HMC can often explore long, skinny modes much more quickly than RWM. Thus, it is possible to tune the depth and length so that the time for RWM to mix on the long mode is much larger than the time to escape the small mode, while the time for HMC to mix on the long mode is much smaller than the time to escape the small mode. In this case, the HMC algorithm can exhibit metastability while the RWM algorithm doesn't.   } \label{FigBanana}
%\end{figure}

More generally, we are interested in determining which ``momentum-based" algorithms satisfy the limit \eqref{LimConj2Strong}. We note that not all similar chains satisfy this equality. In particular, lifted Markov chains can replace the ``1" with a ``$\frac{1}{2}$" for realistic examples (see \textit{e.g.} \cite{bierkens2017piecewise}). It would be valuable to determine if other generic algorithms, such as Algorithm \ref{alg:Riemannian_algorithm} or \cite{bierkens2016zig}, can also achieve this.

Finally, we suggest a more straightforward question. The level sets of the multimodal densities studied in this paper are \textit{completely disconnected} below some (fairly large) threshold. Thus, the multimodality is essentially invisible to the gradient of the log-likelihood within each mode. Can HMC, and especially Riemannian Manifold HMC under an appropriate implementable choice of Riemannian metric, improve on MH for multimodal examples in which the level sets are \textit{connected}, and multimodality is induced by these level sets being merely \textit{very narrow}? Note that this is only possible in two or more dimensions.

\section*{Acknowledgement}
We would like to thank Neil Shephard for asking us the title question.
NSP thanks Gareth Roberts and Andrew Stuart for introducing him to the optimal scaling of MCMC algorithms, and Jeff Rosenthal for references. Part of this work was done when OM was a Ph.D. student at MIT working on this project under the supervision of NSP at Harvard University, and later when he was a postdoctoral researcher with AS at the University of Ottawa. 
We thank these institutions for their hospitality.

\bibliographystyle{alpha}
\bibliography{HMC}

\newpage 

\appendix

\section{Main Theoretical Results}

We prove our main theoretical bounds.

\subsection{Proof of Lemma \ref{LemCountingPossible}} \label{SecCountingPossible}

Define $F \, : \, \mathbb{R}^{2d} \times \mathbb{R} \mapsto \mathbb{R}^{d}$ by
\be \label{EqDefF}
F(p,q,t) = \gamma_{p,q}(t).
\ee
Since $dF$ is full rank at all points, it is immediate that $F$ is transverse to $\partial S$. Since the domain of $F$ does not have a boundary, this implies by the parametric transversality theorem that the function $\gamma_{p,q}(\cdot)$ is transerve to $\partial S$ for almost all values of $(p,q) \in \mathbb{R}^{2d}$. 

Fix $(p,q)$ for which $F(p,q,\cdot)$ is transverse to $\partial S$. Since $\gamma_{p,q}(\cdot)$ is transverse to $\partial S$, the intersection $\gamma_{p,q}([0,T]) \cap \partial S$ must be a zero-dimensional manifold. Since $[0,T]$ is compact, and $\gamma_{p,q}(\cdot)$ is a continuous function by assumption, $\gamma_{p,q}([0,T])$ must also be compact. Since $\partial S$ is closed, $\gamma_{p,q}([0,T]) \cap \partial S$ is also compact. Since $\gamma_{p,q}(\cdot)$ is tranverse to $\partial S$, at all points $0 < t < T$ there exists $\epsilon > 0$ so that  $|\gamma_{p,q}((t- \epsilon, t + \epsilon)) \cap \partial S| \leq 1$ - in other words, points of the intersection $\gamma_{p,q}([0,T]) \cap \partial S$ are isolated. Since the intersection $\gamma_{p,q}([0,T]) \cap \partial S$ is compact and has isolated points, we conclude that $\gamma_{p,q}([0,T]) \cap \partial S$ must be finite (since every open cover of the intersection has a finite subcover).

Putting these facts together, we see that for almost every value of $(p,q) \in \mathbb{R}^{2d}$, we have both:

\begin{enumerate}
\item $F(p,q,\cdot)$ is transverse to $\partial S$, and
\item The set $\gamma_{p,q}([0,T]) \cap \partial S$ is finite.
\end{enumerate}

This completes the proof.

\subsection{Proof of Theorem \ref{thm:1}} \label{SecPfThm1}

The proof is based on the following observations:

\begin{enumerate}
\item  Taking a step of HMC involves constructing an entire solution \textit{path} $\gamma_{p,q} \, : \, [0,T] \mapsto \mathbb{R}^{d}$. Thus, as long as the path is sufficiently well-behaved, we can determine the probability of crossing from $S$ to $S^{c}$ by computing the probability that the size of the intersection $|\gamma_{p,q}([0,T]) \cap \partial S|$ of the path $\gamma_{p,q}$ with the boundary $\partial S$ is odd.
\item The expected \textit{rate} at which random paths cross $\partial S$ is constant.
\end{enumerate} 

These suggest the conductance of HMC can be expressed in terms of the rate at which very short random paths cross $\partial S$. 

\begin{proof} [Proof of Theorem \ref{thm:1}]

Let $(P,Q) \sim \mu_{H}$ and let $(\hat{P},\hat{Q}) \sim \mathbb{Q}$. Define the total positive flux $\Phi^{+}$ across $\partial S$ by
\be \label{eq:a1}
\Phi^{+} := \frac{1}{2} \sum_{n=1,2,3,\ldots} n \, \P(N_{\partial S}(P,Q)=n) = \frac{1}{2} \E(N_{\partial S}(P,Q)).
\ee

Note that, by reversibility and the fact that $\partial S$ has measure 0, we have 

\be 
\P(\gamma_{P,Q}(T) \notin S, Q \in S) = \P(\gamma_{P,Q}(T) \in S, Q \notin S)
\ee 

and
\be \label{eq:a2}
\mathbb{P}(N_{\partial S}(P,Q)=n, Q \in S) = \mathbb{P}(N_{\partial S}(P,Q)=n, Q \in S^c) = \frac{1}{2}\mathbb{P}(N_{\partial S}(P,Q)=n).
\ee 

This implies 

\be \label{eq:a3}
\sum_{n=1,3,5,\ldots} \P( N_{\partial S}(P,Q) = n) &\stackrel{{\scriptsize \textrm{Eq. }}\eqref{eq:a2}}{=}  \sum_{n=1,3,5,\ldots}(\mathbb{P}(N_{\partial S}(P,Q)=n, Q\in S)\\&  \qquad \qquad \qquad \qquad + \mathbb{P}(N_{\partial S}(P,Q)=n, Q \in S^c))\\
& = 2 \sum_{n=1,3,5,\ldots} \P(N_{\partial S}(P,Q) =n, Q \in S),
\ee

This allows us to compute

\be
\mathbb{P}(\gamma_{P,Q}(T) \in S^c, \, Q \in S) &= \sum_{n=1,3,\ldots} \mathbb{P}(N_{\partial S}(P,Q) =n, Q \in S) \\
&\stackrel{{\scriptsize \textrm{Eq. }}\eqref{eq:a3}}{=} \frac{1}{2}\sum_{n=1,3,\ldots} \mathbb{P}(N_{\partial S}(P,Q)=n) \\
&= \frac{1}{2}\sum_{n=1,3,\ldots} \frac{2\Phi^+}{n} \P(N_{\partial S}(\hat{P},\hat{Q})=n) \\
&\stackrel{{\scriptsize \textrm{Eq. }}\eqref{eq:a1}}{=} \Phi^+ \cdot \sum_{n=1,3,\ldots} \frac{1}{n} \P(N_{\partial S}(\hat{P},\hat{Q})=n) \\
&= \Phi^+ \cdot \E \bigg[\frac{1}{N_{\partial S}(\hat{P},\hat{Q})} \cdot \mathbbm{1}\{N_{\partial S}(\hat{P},\hat{Q}) \, \mathrm{odd}\}\bigg].
\ee

Next, we must compute $\Phi^+$. For $0 < a < c < \infty$, define
\be 
N_{\partial S}(p,q,a,c) \equiv |\{ t \in [a,c] \, : \, \gamma_{p,q}(t) \in \partial S\}|.
\ee

Recall that, by the conservation of volume formula in Equation \eqref{EqConsVol}, we have 
\be 
\E[N_{\partial S} (P,Q,a,c)] =  \E[N_{\partial S} (P,Q,0,b-a)] + \E[N_{\partial S} (P,Q,0,c-b)]
\ee 
for all $0 < a < b < c < \infty$; note that all three expectations are well-defined by Lemma \ref{LemCountingPossible}. Viewing $\Phi^{+} = \Phi^{+}(T)$ as a function of the integration time $T \in \mathbb{R}^{+}$, this decomposition formula implies that $\frac{d \Phi^{+}(t)}{dt}$ is constant. We must now compute the derivative

\be  \label{EqDerivCalc}
\frac{d \Phi^{+}(t)}{dt} = \frac{1}{2} \lim_{h \rightarrow 0} h^{-1}  \E [N_{\partial S}(P,Q,0,h)],
\ee

which computes the number of times that very short Hamiltonian paths cross the surface $\partial S$. Since the paths are very short and neither the paths nor $\partial S$ ``bend" very much, linearizing and ignoring multiple crossings suggests 
\be 
\frac{d \Phi^{+}(t)}{dt} = \int_{\partial S} \int_{\mathbb{R}^d}  \mu_{H}(p,q) \cdot \langle v^+(p,q), \eta(q) \rangle \mathrm{d}q,
\ee 

where 
\be  \label{EqDefVPLate}
v^+(p,q) = \gamma_{p,q}'(0) \, \mathbbm{1} \{ \langle \gamma_{p,q}'(0), \eta(q) \rangle \geq 0\}.
\ee 

This turns out to be correct, but we don't know a standard reference for this fact. We summarize it as the following technical lemma, whose proof is deferred to Appendix \ref{AppSubsecTechIntLemmaProof}:

\begin{lemma} \label{LemmaTechIntLemma}
With notation as above,
\be
\Phi^+ = T \cdot \int_{\partial S} \int_{\mathbb{R}^d}  \mu_{H}(p,q) \cdot \langle v^+(p,q), \eta(q) \rangle \mathrm{d}q.
\ee
\end{lemma}

Applying Hamilton's equations \eqref{EqHamEq}, we have 

\be  
\langle v^+(p,q) , \eta(q) \rangle =|p_q|\cdot \mathbbm{1} \{ p \in \mathcal{P}^+_S(q) \},
\ee 
and so by Lemma \ref{LemmaTechIntLemma}
\be\label{eq:32}
\Phi^+ = T \cdot \int_{\partial S} \int_{\mathcal{P}^+_S(q)} \pi(p,q) \cdot |p_q| \mathrm{d}q.
\ee
This completes the proof of Equation \eqref{EqPosFluxConc1}. In the case of Isotropic-Momentum HMC, Equation \eqref{eq:32} simplifies to 

\be
\Phi^+ &= T \cdot \int_{\partial S} \int_{\mathcal{P}^+_S(q)} \pi(p,q) \cdot |p_q| \mathrm{d}q\\
%&= T \cdot \int_{\partial S} \int_0^\infty \pi(q) \cdot f_{\mathcal{N}(0,1)}(y) \cdot y \mathrm{d}y \mathrm{d}q\\
&= T \cdot \int_{\partial S} \int_0^\infty \pi(q)   \frac{1}{\sqrt{2 \pi}} y e^{-\frac{1}{2}y^2} \cdot y \mathrm{d}y \mathrm{d}q\\
%&= T  \cdot \int_{\partial S} \pi(q) \mathrm{d}q \cdot \int_0^\infty  \frac{1}{\sqrt{2 \pi}} y^2 e^{-\frac{1}{2}y^2} \mathrm{d}y\\
&= T  \cdot \int_{\partial S} \pi(q) \mathrm{d}q \cdot \frac{1}{2}. \\
\ee
This completes the proof of the theorem.
\end{proof}

\subsection{Proof of Lemma \ref{LemmaTechIntLemma}} \label{AppSubsecTechIntLemmaProof}

We prove Lemma \ref{LemmaTechIntLemma}, a technical bound used in the proof of Theorem \ref{thm:1}. Before doing so, we need a bound on the difference between the Hamiltonian path $\gamma_{p,q}(t)$ and its obvious linear approximation, defined by 
\be
\hat{\gamma}_{p,q}(t) = \gamma_{p,q}(0) + t \, \gamma_{p,q}'(0).
\ee 

The following bound will be used frequently in the proof of Theorem \ref{thm:1}:

\begin{lemma} \label{LemmaLinApprox}
Fix initial points $(p,q) \in \mathbb{R}^{2d}$, let $H$ be a smooth Hamiltonian,  let $\{(p(t),q(t))\}_{t \geq 0}$ be the solution to Hamilton's equations \eqref{EqHamEq} associated with this Hamiltonian and initial conditions, and define $\gamma_{p,q}$ as in the remainder of the paper.  Assume that there exists some set $\mathcal{X} \subset \mathbb{R}^{d}$ and some constant $0 < C < \infty$ so that $\sup_{p,q \in \mathcal{X}} \| \frac{\partial}{\partial q} H(p,q) \| < C$. Let $\tau_{\mathrm{exit}} = \inf \{ t \geq 0 \, : \, \gamma_{p,q} \notin \mathcal{X}\}$.  Then for all $0 \leq s \leq t \leq \tau_{\mathrm{exit}}$, 

\be \label{IneqLinApproxPosGen}
\max_{s \leq u \leq t} \| \gamma_{p,q}(u) - \hat{\gamma}_{\gamma_{p,q}'(s), \gamma_{p,q}(s)}(u-s) \| \leq \frac{C}{2}(t-s)^{2}
\ee

and 

\be \label{IneqLinApproxMomGen}
\max_{s \leq u \leq t} \| \gamma_{p,q}'(u) - \hat{\gamma}_{p,q}'(s) \| \leq C (t-s).
\ee

\end{lemma}

\begin{proof}
Note that, by reparameterization, we can assume WLOG that $s=0$. By Hamilton's equations, for all $0 \leq t \leq \tau_{\mathrm{exit}}$,
\be
\| \gamma_{p,q}'(t) - \gamma_{p,q}'(0) \| &= \| \int_{0}^{t} \gamma_{p,q}'(s) ds \| \\
&\leq \int_{0}^{t} \| \gamma_{p,q}'(s) \| ds \\
&\leq \int_{0}^{t} \, C ds = Ct.
\ee

This proves Inequality \eqref{IneqLinApproxMomGen}. Applying this bound, we have for $0 \leq t \leq \tau_{\mathrm{exit}}$,

\be
\| \gamma_{p,q}(t) - \hat{\gamma}_{p,q}(t) \| &= \| \int_{0}^{t} (\gamma_{p,q}'(s) - \hat{\gamma}_{p,q}'(s)) ds \| \\
&\leq  \int_{0}^{t} \| \gamma_{p,q}'(s) - \hat{\gamma}_{p,q}'(s) \| ds \\
&\leq \int_{0}^{t} C s ds = \frac{C}{2} t^{2}.
\ee
This proves Inequality \eqref{IneqLinApproxPosGen} and completes the proof of the lemma.
\end{proof}

We now prove Lemma \ref{LemmaTechIntLemma}:

\begin{proof} [Proof of Lemma \ref{LemmaTechIntLemma}]

Our proof strategy is to break $\mathbb{R}^{2d}$ into pieces, and then estimate the contributions of each piece to the derivative \eqref{EqDerivCalc}. We repeat the observation that the result of this calculation is exactly what one might expect from \textit{e.g.} taking naive Taylor expansions at this point and ignoring the possibility that any curve will intersect $S$ more than once.

Define the nearly-disjoint cover $\mathfrak{R} = \{ \mathbb{R}^{d} \times [a_{1},a_{1}+1] \times[a_{2},a_{2}+1] \times \ldots \times [a_{d},a_{d}+1]  \, : \, a_{1},\ldots,a_{d} \in \mathbb{Z}\}$ of $\mathbb{R}^{2d}$, whose elements look like rectangular ``slabs" that have side width 1 in their first $d$ dimensions and ``infinite" side width in their last $d$ dimensions. By the monotone convergence theorem, we can rewrite Equation \eqref{EqDerivCalc} as  

\be  \label{EqBreakUpPhiDeriv}
\frac{d \Phi^{+}(t)}{dt} = \frac{1}{2} \lim_{h \rightarrow 0} h^{-1}  \E [N_{\partial S}(P,Q,0,h)] = \frac{1}{2}  \sum_{A \in \mathfrak{R}} \lim_{h \rightarrow 0} h^{-1} \E [N_{\partial S}(P,Q,0,h) \, \mathbbm{1}_{(P,Q) \in A }] \\ 
\ee

We now estimate the terms on the right-hand side of this formula. Fix $A \in \mathfrak{R}$. For $(p,q) \in \mathbb{R}^{2d}$ and $t \in \mathbb{R}$, recall that $\hat{\gamma}_{p,q}(t) = \gamma_{p,q}(0) + t \, \gamma_{p,q}'(0) \in \mathbb{R}^{d}$ is the usual linear approximation to the path $\gamma_{p,q}(t)$. We then define 
\be 
\hat{N}_{\partial S}(p,q,a,b) := | \{t \in [a,b] \, : \, \hat{\gamma}_{p,q}(t) \in \partial S\} | 
\ee
when the set on the right-hand side is finite. By exactly the same argument as in the proof of Lemma \ref{LemCountingPossible}, the path $\hat{\gamma}_{p,q}(t)$ is transverse to $\partial S$ for almost every value of $(p,q) \in \mathbb{R}^{2d}$. As we are only interested in calculating integrals, we can ignore sets of measure 0 and so restrict our attention to values of $(p,q)$ for which $\gamma_{p,q}$ and $\hat{\gamma}_{p,q}$ are transverse to $\partial S$ over their entire paths.

We will now show that $\hat{N}_{\partial S}(p,q,0,h)$ agrees with $N_{\partial S}(p,q,0,h)$ outside of a set of measure $o(h)$, which will turn out to be negligible. Assume for the remainder of the proof that $h < 0.1$ and define $\epsilon(h) = h^{0.95}$. Although our argument involves checking several cases, all of our estimates are based on the following two heuristics: 
\begin{enumerate}
\item For $h$ small we have $\P[ \|P \| > \epsilon(h)] = o(h^{2})$. Thus, we can ignore everything outside of a ball of radius $\epsilon(h)$ around a point $q$ of interest. 
\item Within a small ball of radius $\epsilon(h)$ around $q$, the path $\gamma_{p,q}$ is close to its linear approximation $\hat{\gamma}_{p,q}$, and similarly the surface $\partial S$ is close to its linear approximation $\mathcal{T}_{x}$. In particular, Lemma \ref{LemmaLinApprox} and Assumption \ref{DefLocallyWellBehaved} tell us that we can approximate $\gamma_{p,q}$ and $\partial S$ by $\hat{\gamma}_{p,q}$ and $\mathcal{T}_{x}$ respectively, up to $O(\epsilon(h)^{2})$ errors. 
\end{enumerate}

Most of the proof is checking that the $O(\epsilon(h)^{2})$ deviations in \textit{paths} that occur in the second heuristic can be translated into $O(\epsilon(h)^{2})$ bounds on possible \textit{initial conditions} $p,q$.

We now begin bounding the measure of the set
\be 
\{ (p,q) \in A \, : \, N_{\partial S}(p,q,0,h) \neq \hat{N}_{\partial S}(p,q,0,h) \}
\ee 
for $h$ small by noting that $N_{\partial S}(p,q,0,h) = \hat{N}_{\partial S}(p,q,0,h)$ unless at least one of the following ``bad events" occurs: 

\begin{defn}\label{DefBadEvents}
\begin{enumerate}
\item \textbf{The Hamiltonian path is unusually fast:} that is,
\be 
\mathcal{E}_{1} = \{ (p,q) \in A \, : \, \sup_{ 0 \leq t \leq h} \|  \gamma_{p,q}'(t) \| \geq h^{-1} \epsilon(h) \}.
\ee
\item \textbf{The approximate path just barely misses the boundary:} that is,
\be 
\mathcal{E}_{2}  = \{ (p,q) \in A \, : \, N_{\partial S} (p,q,0,h) > 0, \, \hat{N}_{\partial S}(p,q,0,h) = 0  \}.
\ee 

\item \textbf{The Hamiltonian path just barely misses the boundary:} that is,
\be
\mathcal{E}_{3}  = \{ (p,q) \in A \, : \, N_{\partial S} (p,q,0,h) = 0, \, \hat{N}_{\partial S}(p,q,0,h) > 0  \}.
\ee

\item \textbf{One or both paths intersect the boundary more than once:}  that is,
\be 
\mathcal{E}_{4}  = \{ (p,q) \in A \, : \, \max(N_{\partial S}(p,q,0,h), \hat{N}_{\partial S }(p,q,0,h)) >  1  \}.
\ee 
\end{enumerate}

\end{defn}

More carefully, for $(p,q) \in A \backslash (\mathcal{E}_{1} \cup \mathcal{E}_{2} \cup \mathcal{E}_{3} \cup \mathcal{E}_{4})$, we have $N_{\partial S}(p,q,0,h) = \hat{N}_{\partial S}(p,q,0,h) \in \{0,1\}$. See Figure \ref{FigNearlyParallel} for a quick illustration of why we don't expect short paths to intersect $\partial S$ more than once.

\begin{figure}[H]
\includegraphics[scale=0.4]{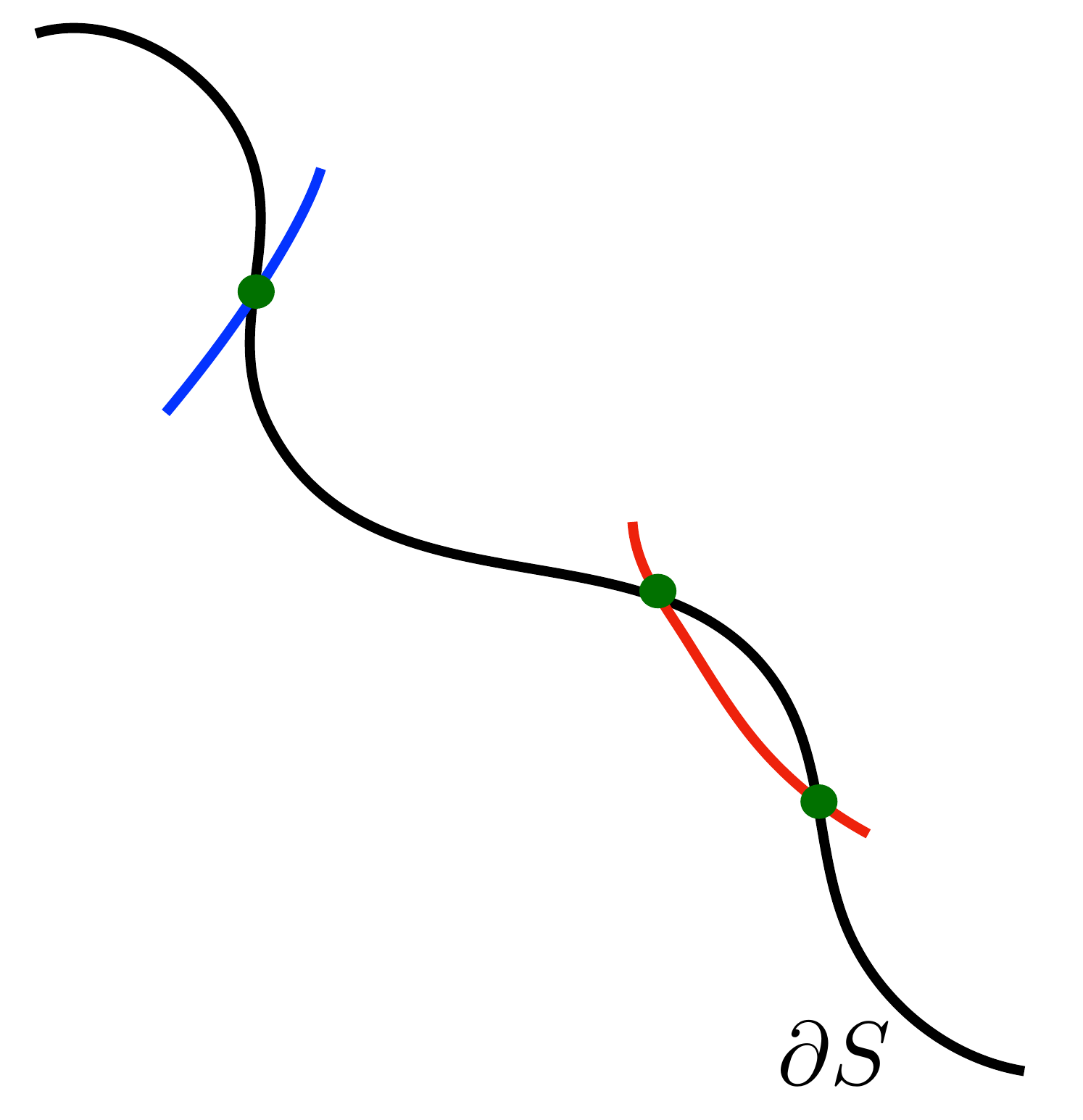}
\caption{A boundary $\partial S$ is shown in black. The blue trajectory intersects only once; the red trajectory hits a highly-curved region more than once. By Assumption \ref{DefLocallyWellBehaved}, these ``highly-curved" regions are not possible for very short paths. \label{FigMultiple}}
\end{figure}

We now show that, for $i \in \{1,2,3,4\}$,
\be \label{IneqManySetsDontContribute}
\P[\mathcal{E}_{i}] = o(h).
\ee
This will allow us to ignore these events when calculating the expectation in Equation \eqref{EqDerivCalc}.

Most of the remainder of this proof is dedicated to proving Inequality \eqref{IneqManySetsDontContribute}. We state a few conventions that we use throughout this proof. We will often say that an inequality ``holds for $h$ sufficiently small." This is always shorthand for the longer phrase: ``there exists some $H_{0} = H_{0}(A)$ so that the following inequality holds for all $h < H_{0}$." In particular, these bounds will always hold uniformly in $A$, but will generally not hold uniformly in $\mathbb{R}^{2d}$. We will use $c,C,C',$ etc as generic constants. Generally, we use $c,c'$, etc for a constant that appears only in the next line, while $C, C'$,etc are used for constants that persist over a longer section of the proof.

We begin by showing that $\mathcal{E}_{1}$ has negligible probability:

\begin{prop}
With notation as above, 
\be  \label{IneqE1Small}
\P[(P,Q) \in \mathcal{E}_{1} ] = o(h^{2}).
\ee 
\end{prop}

\begin{proof}
Define $\tau_{\mathrm{far}}(p,q) = \min(h, \inf \{t \geq 0 \, : \, \gamma_{p,q}(t) \notin A_{10}\})$. Since the projection of $A_{10}$ onto the position space is contained in a compact subset of $\mathbb{R}^d$ and $\nabla U$ is continuous, $\|\nabla U(x)\|$ must be bounded on $A_{10}$. Therefore, by Lemma \ref{LemmaLinApprox} and the assumption that $S$ satisfies Assumption \ref{DefLocallyWellBehaved}, 
\be 
\sup_{0 \leq t \leq \tau_{\mathrm{far}} (p,q)} \| \gamma_{p,q}'(t) \| \leq \| \gamma_{p,q}'(0) \| + c_{1} h
\ee
for some constant $0 < c_{1} < \infty$ that does not depend on the starting point $(p,q) \in A$. Since $\inf_{x \in A, \, y \notin A_{10}} \|x - y \| = 10$, we have for all $0< h \ll c_{1}^{-2}$ sufficiently small that 
\be
\P[ \{ (P,Q) \in A\} \cap \{ \sup_{ 0 \leq t \leq h} \|  \gamma_{P,Q}'(t) \| \geq h^{-1} \epsilon(h)  \}] &\leq \P[ \{(P,Q) \in A\} \cap \{ \| \gamma_{P,Q}'(0) \| \geq h^{-1} \epsilon(h) - c_{1} h\}] \\
&\leq \P[  \| \gamma_{P,Q}'(0) \| \geq h^{-1} \epsilon(h) - c_{1} h\}] \\
&\leq e^{- \Omega(h^{-2} \epsilon(h))} =o(h^{2}),
\ee
where in the case of RHMC the second-last line uses Assumption \ref{AssumptionBoundFIM}.

Thus, we conclude
\be  
\P[(P,Q) \in \mathcal{E}_{1} ] = o(h^{2}),
\ee 
completing the proof.
\end{proof}

Next, we estimate the probability of $\mathcal{E}_{2} \backslash \mathcal{E}_{1}$. 

\begin{prop}
With notation as above,
\be  \label{IneqE2Small}
\P[ (P,Q) \in \mathcal{E}_{2} \backslash \mathcal{E}_{1}] = o(h).
\ee 
\end{prop}

\begin{proof}

Roughly speaking, this event has small probability because it can only occur if $\hat{\gamma}_{p,q}(h)$ is very close to $\partial S$. To make this intuition rigorous, define 
\be 
\mathcal{C}(h) = \{ q \in \mathbb{R}^{d} \, : \, (0,q) \in A, \, q \in (\partial S)_{\epsilon(h)} \}
\ee
to be the second coordinates of the elements of $A$ that are near $\partial S$. For $q \in \mathcal{C}(h)$, let 
\be 
s(q) = \mathrm{argmin}_{s \in \partial S} ( \|q - s \|)
\ee 
be the closest element of $\partial S$ to $q$. Note that, by Assumption \ref{DefLocallyWellBehaved} and the fact that $\partial S \cap A$ is compact, $\sup_{q \in \mathcal{C}(h)} |s(q)| = 1$ for all $h$ sufficiently small; thus by a slight abuse of notation we treat $s(q)$ as if it were a point rather than a set. Let $v(q) = \frac{s(q) - q}{\|s(q) - q\|}$ be the unit vector from $q$ to $s(q)$.

For a set $B \subset \mathbb{R}^{d}$ and a point $x \in \mathbb{R}^{d}$, define the usual notion of set translation
\be 
B + x = \{ b + x \, : \, b \in B\}.
\ee 

By Lemma \ref{LemmaLinApprox} and Assumption \ref{DefLocallyWellBehaved}, we have the containment condition
\be  \label{SomeContaintmentBoundingE2}
(\partial S) \cap (\mathcal{C}(h)_{10 \epsilon(h)}) \subset \cup_{- C \epsilon(h)^{2} \leq \delta \leq C \epsilon(h)^{2}} (\mathcal{T}_{s(q)} + \delta \, v(q)) 
\ee
for some $0 < C < \infty$. That is, the part of $\partial S$ contained in $\mathcal{C}(h)_{10 \epsilon(h)}$ is sandwiched between the two affine planes 

\be 
H_{1}(C) = \mathcal{T}_{s(q)} - C \epsilon h^{2} \, v(q), \, H_{2}(C) = \mathcal{T}_{s(q)} + C \epsilon h^{2} \, v(q). 
\ee

Next, fix $(p,q) \in  \mathcal{E}_{2} \backslash \mathcal{E}_{1}$ satisfying 
\be \label{EqFirstContainmentJan}
\gamma_{p,q}([0,h]) \cap \partial S \cap \mathcal{C}(h)_{10 \epsilon(h)} \neq \emptyset.
\ee 
Since the path $\gamma_{p,q}([0,h])$ passes through $\partial S \cap \mathcal{C}(h)_{10 \epsilon(h)}$, and that set is sandwiched between two hyperplanes $H_{1}(C)$, $H_{2}(C)$ by the containment bound \eqref{SomeContaintmentBoundingE2}, the path must either stay entirely between the two hyperplanes or must intersect one of them - in particular, for all $C' \geq C$, it must satisfy at least one of:

\begin{enumerate}
\item $\gamma_{p,q}([0,h]) \cap H_{1}(C') \neq \emptyset$, or 
\item $\gamma_{p,q}([0,h]) \cap H_{2}(C') \neq \emptyset$, or
\item $\gamma_{p,q}([0,h]) \subset \cup_{- C' \epsilon(h)^{2} \leq \delta \leq C' \epsilon(h)^{2}} (\mathcal{T}_{s(q)} + \delta \, v(q))$.
\end{enumerate}

Thus, applying Lemma \ref{LemmaLinApprox} again, for all $C'' \gg C' $ sufficiently large compared to $C' \geq C$ we must have that at least one of the following hold:

\begin{enumerate}
\item $\hat{\gamma}_{p,q}([0,h]) \cap H_{1}(C'') \neq \emptyset$, or 
\item $\hat{\gamma}_{p,q}([0,h]) \cap H_{2}(C'') \neq \emptyset$, or
\item $\hat{\gamma}_{p,q}([0,h]) \subset \cup_{- C'' \epsilon(h)^{2} \leq \delta \leq C'' \epsilon(h)^{2}} (\mathcal{T}_{s(q)} + \delta \, v(q))$.
\end{enumerate}
For $r > 0$, denote by $B_{r} = B_{r}(q)$ the ball of radius $r$ around $q$. Note that, by Assumption \ref{DefLocallyWellBehaved}, for all $0 < r < R_{0} = R_{0}(A)$ sufficiently small, the surface $(\partial S) \cap B_{r}$ must separate $H_{1}(C'') \cap  B_{r}$ from $H_{2}(C'') \cap  B_{r}$, as shown in Figure \ref{FigSepPic}:

\begin{figure}[H]
\includegraphics[scale=0.5]{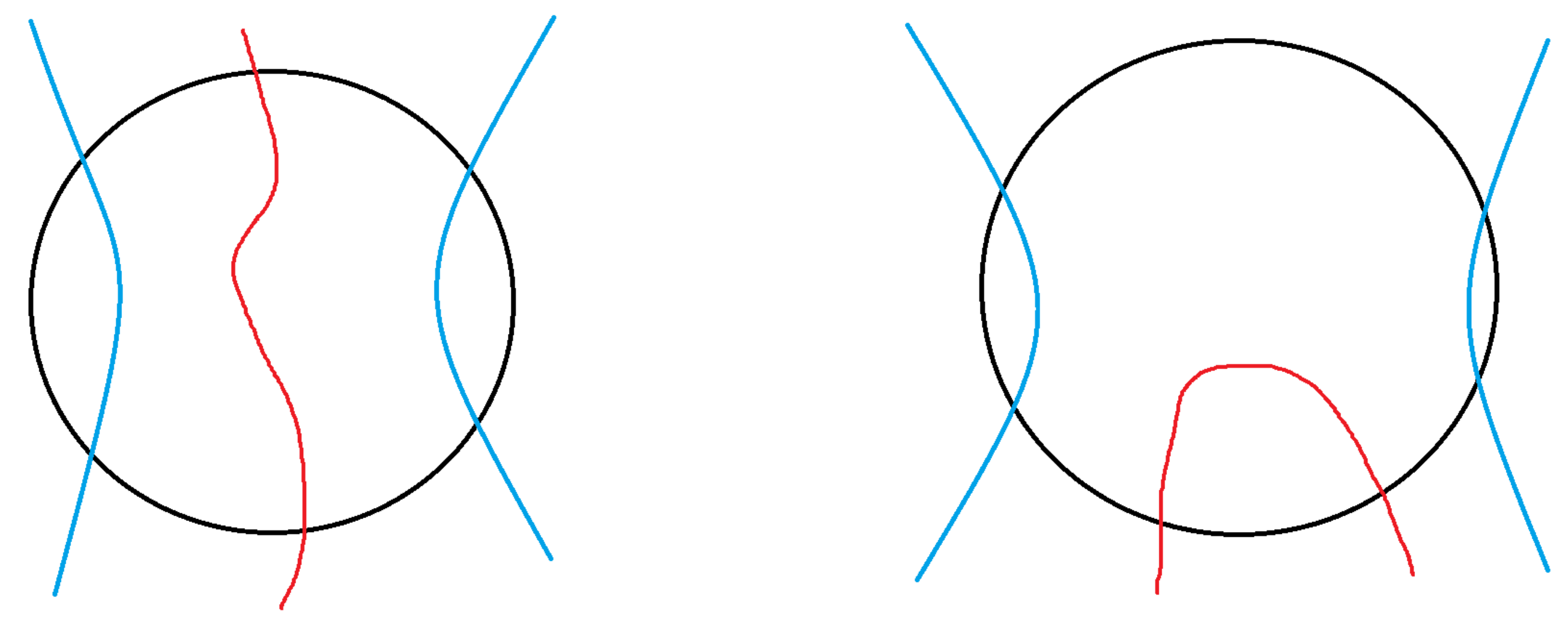}
\caption{We see $\partial S$ in red, $B_{r}$ in black, and $H_{1}, H_{2}$ in blue. The right-hand side picture is only possible for $r$ large, as it would require the normal vector of $\partial S$ to have a derivative of size that goes to infinity as $r$ goes to 0. } \label{FigSepPic}
\end{figure}

 Thus, the containment condition \eqref{SomeContaintmentBoundingE2} implies that $\hat{\gamma}_{p,q}$ cannot pass through both hyperplanes $H_{1}(C'')$, $H_{2}(C'')$ (or else it would intersect $\partial S$). Thus, for the same value of $C''$, we must have  

\be 
\min( |\hat{\gamma}_{p,q}([0,h]) \cap H_{1}(C'')|, \,  |\hat{\gamma}_{p,q}([0,h]) \cap H_{2}(C'')|) = 0.
\ee 

Summarizing, we have found that $\hat{\gamma}_{p,q}([0,h])$ must either:

\begin{enumerate}
\item remain sandwiched between $H_{1}(C'')$ and $H_{2}(C'')$, \textit{or}
\item intersect with \textit{exactly one} of $H_{1}(C'')$, $H_{2}(C'')$.
\end{enumerate}

This implies that at least one endpoint of $\hat{\gamma}_{p,q}([0,h])$ must also be sandwiched between $H_{1}(C'')$ and $H_{2}(C'')$. More precisely, if we define $W =  \cup_{- C'' \epsilon(h)^{2} \leq \delta \leq C'' \epsilon(h)^{2}} (\mathcal{T}_{s(q)} + \delta \, v(q))$, we have shown that at least one of the following must hold:

\begin{enumerate}
\item $q \in  W$, \textit{or}
\item $\hat{\gamma}_{p,q}(h) \in  W$.
\end{enumerate}

Thus, for all $h$ sufficiently small, 

\be \label{IneqBoundE2QCloseToBoundary}
\P[ \{Q \in \mathcal{C}(h)\} \cap \{ (P,Q) \in \mathcal{E}_{2} \backslash \mathcal{E}_{1} \}]  &\leq \P[Q \in \mathcal{C}(h) \cap W] \\
& \qquad + \P[ \{ \hat{\gamma}_{P,Q}(h) \in W \} \cap \{ Q \in \mathcal{C}(h)\} \cap \{(P,Q) \in \mathcal{E}_{2} \backslash \mathcal{E}_{1}\}] \\
&\leq \P[Q \in \mathcal{C}(h) \cap W]  + \P[ Q \in \mathcal{C}(h)] \times  \P[\hat{\gamma}_{(P,Q)}(h) \in W | Q \in \mathcal{C}(h)] \\
&\leq O(\epsilon(h)^{2}) + \P[ Q \in \mathcal{C}(h)] \sup_{q \in \mathcal{C}(h)} \P[\hat{\gamma}_{(P,q)}(h) \in W] \\
&= O( \epsilon(h)^{2}) + O( \epsilon(h))\,  \sup_{q \in \mathcal{C}(h)} \P[\hat{\gamma}(P,q) \in W] , \\
\ee
where the third inequality holds by independence of $P$ and $Q$.

To bound the last term, note that $\hat{\gamma}(P,q)$ is a $d$-dimensional Gaussian with covariance matrix of norm $O(h^{2})$ (by Assumption \ref{AssumptionBoundFIM} in the case of RHMC), while $W$ is the product of a hyperplane with an interval of length $C''\, \epsilon(h)^{2}$. Thus, if we let $Z$ be the standard $d$-dimensional Gaussian, we have by rescaling and rotation: 
\begin{align*}
\P[\hat{\gamma}(P,q) \in W] &\leq \P[Z \in [-C''' h^{-1} \epsilon(h)^{2},C''' h^{-1} \epsilon(h)^{2}] \times \mathbb{R}^{d-1}]  \\
&= \P[|Z[1]| \leq C''' h^{-1} \epsilon(h)^{2}] = O(h^{-1} \epsilon(h)^{2}),
\end{align*}
where the constant $C''' = C''$ in the case of isotropic HMC but may be larger in the case of RHMC. Combining this bound with Inequality \eqref{IneqBoundE2QCloseToBoundary}, we conclude 
\be  \label{IneqBoundE2NearFinal1}
\P[ \{Q \in \mathcal{C}(h)\} \cap \{ (P,Q) \in \mathcal{E}_{2} \backslash \mathcal{E}_{1} \}] = O(h^{-1} \epsilon(h)^{3}) = o(h).
\ee

Finally, we consider the case that $(p,q) \in \mathcal{E}_{2} \backslash \mathcal{E}_{1}$ but $q \notin \mathcal{C}(h)$. In this case, Lemma \ref{LemmaLinApprox} implies  
\be 
\| \gamma_{p,q}'(0) \| \geq c h^{-1} \epsilon(h)
\ee 
for some $c > 0$ that does not depend on $h$ (since otherwise $\gamma_{p,q}([0,h])$ would not intersect $\partial S$). Thus,
\be  \label{IneqBoundE2NearFinal2}
\P[ \{Q \notin \mathcal{C}(h)\} \cap \{ (P,Q) \in \mathcal{E}_{2} \backslash \mathcal{E}_{1} \}] \leq \P[\| \gamma_{P,Q}'(0) \| \geq c h^{-1} \epsilon(h)] = o(h),
\ee 
where again the last line uses Assumption \ref{AssumptionBoundFIM} in the case of RHMC. 

Combining Inequalities \eqref{IneqBoundE2NearFinal1} and \eqref{IneqBoundE2NearFinal2} completes the proof.

\end{proof}

The proof that 
\be  \label{IneqE3Small}
\P[(P,Q) \in\mathcal{E}_{3} \backslash \mathcal{E}_{1}] = o(h)
\ee
is essentially the same; we omit the details.

Finally, we bound $\P[\mathcal{E}_{4}]$:
\begin{prop}
With notation as above,
\be  \label{IneqE4Small}
\P[(P,Q) \in \mathcal{E}_{4}] = o(h).
\ee
\end{prop}

\begin{proof}
We will check that both $\hat{N}_{\partial S}(P,Q,0,h) \leq 1$ and $N_{\partial S}(P,Q,0,h) \leq 1$ with probability at least $1 - o(h)$.

We begin by bounding $\hat{N}_{\partial S}(P,Q,0,h)$, since it will be simpler. We consider two cases: 

\begin{enumerate}
\item \textbf{$\hat{\gamma}_{p,q}$ travels very far:} that is,
\be 
\widehat{\mathcal{E}}_{4,1} = \{ (p,q) \in A \, : \, \hat{N}_{\partial S}(p,q,0,h) > 1, \,   \| \gamma_{p,q}'(0) \| \geq \epsilon(h) \, h^{-1} \}.
\ee 
\item  \textbf{$\hat{\gamma}_{p,q}$ must intersect $\partial S$ in two somewhat nearby positions:} that is,
\be 
\widehat{\mathcal{E}}_{4,2} = \{ (p,q) \in A \, : \, \hat{N}_{\partial S}(p,q,0,h) > 1, \,  \| \gamma_{p,q}'(0) \| < \epsilon(h) \, h^{-1} \}.
\ee
\end{enumerate}

We note immediately that
\be  \label{IneqEasyEps41}
\P[(P,Q) \in \widehat{\mathcal{E}}_{4,1}] \leq \P[\|\gamma_{P,Q}'(0)\| \geq h^{-1} \epsilon(h)] = o(h^{2}),
\ee
where again the last line uses Assumption \ref{AssumptionBoundFIM} in the case of RHMC.

Next, we estimate the probability of $\widehat{\mathcal{E}}_{4,2}$. Let $(p,q) \in \widehat{\mathcal{E}}_{4,2}$. Define the first collision times 
\be
\tau_{\mathrm{hit}}(p,q) &= \inf \{t \geq 0 \, : \, \gamma_{p,q}(t) \in \partial S \}  \\
\hat{\tau}_{\mathrm{hit}}(p,q) &= \inf \{t \geq 0 \, : \, \hat{\gamma}_{p,q}(t) \in \partial S \}.
\ee

By Assumption \ref{DefLocallyWellBehaved}, there exists $C > 0$ so that
\be  \label{ContainmentAlsoNHatStarter}
(\partial S \cap A_{10}) \subset \hat{\gamma}_{p,q}(\hat{\tau}_{\mathrm{hit}}(p,q)) + \widehat{W}'(p,q),
\ee  
where $\widehat{W}'(p,q)$ is the envelope with quadratic boundary:
\be
\widehat{W}'(p,q) = \{ x \in \mathbb{R}^{d} \, : \, \| \mathbbm{Proj}_{\hat{\gamma}_{p,q}(\hat{\tau}_{\mathrm{hit}}(p,q))} (x) - x \| \leq C  \| \mathbbm{Proj}_{\hat{\gamma}_{p,q}(\hat{\tau}_{\mathrm{hit}}(p,q))} (x) - \hat{\gamma}_{p,q}(\hat{\tau}_{\mathrm{hit}}(p,q)) \|^{2} \}.
\ee

 Note that, since the boundary of $\widehat{W}'$ consists of two quadratic functions, we have for all sufficiently small $h$ that there exists $C_{1} > 0$ so that
\be 
\widehat{W}'(p,q) \cap \{x \, : \, \|q- x \| \leq \epsilon(h)\} \subset \widehat{W}''(p,q),
\ee 
where  
\be 
\widehat{W}''(p,q) = \{ \hat{\gamma}_{p,q}(\hat{\tau}_{\mathrm{hit}}(p,q)) + v \, : \,  \frac{|\langle v, \eta(\hat{\gamma}_{p,q}(\hat{\tau}_{\mathrm{hit}}(p,q))) \rangle|}{ \|v \| }  \leq C_{1} \epsilon(h)^{2} \}
\ee
is a linear envelope in the neighbourhood of the intersection point $\hat{\gamma}_{p,q}(\hat{\tau}_{\mathrm{hit}}(p,q))$. Thus, 
\be  \label{IneqAngleCondHatN42}
 \frac{|\langle \gamma_{p,q}'(0), \eta(\hat{\gamma}_{p,q}(\hat{\tau}_{\mathrm{hit}}(p,q))) \rangle|}{\|  \gamma_{p,q}'(0) \|} \leq C_{1}' \epsilon(h),
\ee 
for some constant $C_{1}' = C_{1}(1 + o(1))$.

However, by Assumption \ref{DefLocallyWellBehaved}, we have for all fixed $c > 0$

\be 
\sup_{x\in \partial S: \|x - q \| \leq c \epsilon(h)} \| \eta(x) - \eta(s(q)) \| = O(\epsilon(h)).
\ee

Combining this bound with Inequality \eqref{IneqAngleCondHatN42} and the fact that $\|s(q) - \hat{\gamma}_{p,q}(\hat{\tau}_{\mathrm{hit}}(p,q)) \| = O( \epsilon(h))$, there exists $C_{2} > 0$ so that 
\be \label{eq:b1}
\frac{| \langle \gamma_{p,q}'(0), \eta(s(q)) \rangle |}{\|  \gamma_{p,q}'(0) \|} \leq C_{2} \epsilon(h)
\ee 
for all $(p,q) \in \widehat{\mathcal{E}}_{4,2} \backslash \widehat{\mathcal{E}}_{4,1}$. Using the bound $\| \gamma_{p,q}'(0) \| \leq h^{-1} \epsilon(h)$ for $(p,q) \in \widehat{\mathcal{E}}_{4,2} \backslash \widehat{\mathcal{E}}_{4,1}$, we have:

\be  \label{IneqPNearOrthoNHat}
| \langle \gamma_{p,q}'(0), \eta(s(q)) \rangle | \leq C_{2} h^{-1} \epsilon(h)^{2}
\ee

for all $(p,q) \in \widehat{\mathcal{E}}_{4,2} \backslash \widehat{\mathcal{E}}_{4,1}$. Applying Assumption  \ref{DefLocallyWellBehaved} we can strengthen this to the uniform bound  
\be 
\sup_{ q' \in A \, : C_{2} \, \|q' - q \| \leq 10 \epsilon(h)} | \langle  \gamma_{p,q}'(0), \eta(s(q')) \rangle | \leq C_{3} h^{-1} \epsilon(h)^{2}
\ee
for some (perhaps larger) $C_{3} > 0$.
 However, this also implies that $\gamma_{p,q}'(0)$ must have a very small velocity in the direction of $\partial S$, and so applying Lemma \ref{LemmaLinApprox} gives 
\be \label{IneqQVeryCloseNHat}
\inf_{q'' \in \partial S} \| q - q'' \| = O(h^{-1} \epsilon(h)^{2}) + O(\epsilon(h)) = O(h^{-1} \epsilon(h)^{2}).
\ee 

Combining Inequalities \eqref{IneqPNearOrthoNHat} and \eqref{IneqQVeryCloseNHat}, there exists $0 < C_{4} < \infty$ so that 
\be  \label{EqMainCalcConcHatE4}
| \langle \gamma_{p,q}'(0), \eta(s(q)) \rangle | \leq C_{4} h^{-1} \epsilon(h)^{2}, \,  q \in (\partial S)_{C_{4} h^{-1} \epsilon(h)^{2}}
\ee 
both hold for all $(p,q) \in \widehat{\mathcal{E}}_{4,2} \backslash \widehat{\mathcal{E}}_{4,1}$. Setting 
\be 
\mathcal{A}(h) = \{ q \in (\partial S)_{C_{4} h^{-1} \epsilon(h)^{2}} \, : \, (0,q) \in A\}
\ee 
to simplify the notation in the following calculation, Condition \eqref{EqMainCalcConcHatE4} gives
\be
\P[(P,Q) \in \widehat{\mathcal{E}}_{4,2} \backslash \widehat{\mathcal{E}}_{4,1}] &\leq \P[ \{ | \langle  \gamma_{P,Q}'(0), \eta(s(Q)) \rangle | \leq C_{4} h^{-1} \epsilon(h)^{2} \} \cap \{ Q \in \mathcal{A}(h) \}] \\
&\leq \int_{q \in \mathcal{A}(h)} \P[| \langle  \gamma_{P,q}'(0), \eta(s(q)) \rangle | \leq C_{4} h^{-1} \epsilon(h)^{2}] \pi(q) \mathrm{d}q \\
&=O \left( \int_{q \in \mathcal{A}(h)} \, h^{-1} \epsilon(h)^{2} \, \pi(q) \mathrm{d}q \right)\\
&= O( h^{-1} \epsilon(h)^{2} \P[Q \in  \mathcal{A}(h) ]) \\
&= O(h^{-2} \epsilon(h)^{4}),
\ee
where in the last line we use the fact that $\sup_{ q \in \mathcal{A}(1)} \pi(q) < \infty$ (this follows from the fact that $\pi$ is continuous and $\mathcal{A}(1)$ is bounded).

Combining this with Inequality \eqref{IneqEasyEps41}, we conclude 
\be \label{IneqNoBadHats}
\P[\sup_{x\in \partial S: \|x - q \| \leq \epsilon(h)}, \hat{N}_{\partial S}(P,Q,0,h) > 1] = o(h).
\ee

We must next show the analogous bound for $N_{\partial S}(P,Q,0,h)$. The proof will be very similar, but with slightly wider envelopes to allow for the fact that $\gamma_{p,q}$ can be very slightly curved. Since the notation for this argument is fairly involved, we give all of the details;  we will keep closely analogous notation (\textit{e.g.} using the same generic constants for the analogous bounds).

As before, if $N_{\partial S}(p,q,0,h) > 1$, at least one of the following must occur:

\begin{enumerate}
\item \textbf{$\gamma_{p,q}$ travels very far:} we define this bad set to be 
\be 
\mathcal{E}_{4,1} = \{ (p,q) \in A \, : \, N_{\partial S}(p,q,0,h) > 1, \,   \max_{0 \leq t \leq h} \max( \|\gamma_{p,q}(t) - q \|,\,  h \, \epsilon(h)^{-1} \|\gamma_{p,q}'(t) \|) \geq  \epsilon(h)  \}.
\ee 
\item  \textbf{$\gamma_{p,q}$ must intersect $\partial S$ in two somewhat nearby positions:} we define this bad set to be
\be 
\mathcal{E}_{4,2} = \{ (p,q) \in A \, : \, N_{\partial S}(p,q,0,h) > 1, \,   \max_{0 \leq t \leq h} \max( \|\gamma_{p,q}(t) - q \|,\,  h \, \epsilon(h)^{-1} \|\gamma_{p,q}'(t) \|) <  \epsilon(h)  \}.
\ee
\end{enumerate}

We note by Lemma \ref{LemmaLinApprox} that there exists some $c > 0$ so that
\be 
\P[(P,Q) \in \mathcal{E}_{4,1}] \leq \P[\|\gamma_{P,Q}'(0)\| \geq c h^{-1} \epsilon(h)]
\ee
for all $h$ sufficiently small. Thus, applying Assumption \ref{AssumptionBoundFIM} in the case of RHMC,
\be  \label{IneqEasyEps41PartTwo}
\P[(P,Q) \in \mathcal{E}_{4,1}] = O(h^{2}).
\ee

Next, we estimate the probability of $\mathcal{E}_{4,2}$. Let $(p,q) \in \mathcal{E}_{4,2}$.  By Assumption \ref{DefLocallyWellBehaved}, there exists $C > 0$ so that
\be \label{ContainmentAlsoNHatStarter}
(\partial S \cap A_{10}) \subset \gamma_{p,q}(\tau_{\mathrm{hit}}(p,q)) + W'(p,q),
\ee 
where $W'(p,q)$ is the envelope with quadratic boundary:
\be
W'(p,q) = \{ x \in \mathbb{R}^{d} \, : \, \| \mathbbm{Proj}_{\gamma_{p,q}(\tau_{\mathrm{hit}}(p,q))} (x) - x \| \leq C  \| \mathbbm{Proj}_{\gamma_{p,q}(\tau_{\mathrm{hit}}(p,q))} (x) - \gamma_{p,q}(\tau_{\mathrm{hit}}(p,q)) \|^{2} \}.
\ee

Since the boundary of $W'$ consists of two quadratic functions, there exists some constant $C_{1} > 0$ so that, for all sufficiently small $h$, 
\be \label{EqContWP}
W'(p,q) \cap \{x \, : \, \|q- x \| \leq \epsilon(h)\} \subset W''(p,q),
\ee
where  
\be
W''(p,q) = \{ \gamma_{p,q}(\tau_{\mathrm{hit}}(p,q)) + v \, : \,  \frac{| \langle v, \eta(\gamma_{p,q}(\tau_{\mathrm{hit}}(p,q))) \rangle |}{ \|v \| }  \leq C_{1} \epsilon(h)^{2} \}
\ee
is a linear envelope in the neighbourhood of the intersection point $\gamma_{p,q}(\tau_{\mathrm{hit}}(p,q))$. 

We now make our first change from our proof of the bound on $\hat{N}_{\partial S}$. Before finding the analogue to Inequality \eqref{IneqAngleCondHatN42}, we bound how much $\gamma_{p,q}(t)$ can deviate from the direction $\gamma_{p,q}'(\tau_{\mathrm{hit}}(p,q))$ it was taking when it crossed $\partial S$. By Lemma \ref{LemmaLinApprox},  there exists $C_{1}'$ so that 

\be \label{IneqSmallDevIntersectionDirection}
\sup_{\tau_{\mathrm{hit}}(p,q) \leq t \leq h} \| \gamma_{p,q}(t) - (\gamma_{p,q}(\tau_{\mathrm{hit}}(p,q)) + \gamma_{p,q}'(\tau_{\mathrm{hit}}(p,q)) \, (t - \tau_{\mathrm{hit}}(p,q))) \| \leq C_{1}' h^{2}.
\ee 

Thus, combining Inequalities \eqref{EqContWP} and \eqref{IneqSmallDevIntersectionDirection}, there exists $C_{1}'' > 0$ so that
\be  \label{IneqAngleCondNoHatN42}
 \frac{|\langle \gamma_{p,q}'(0), \eta(\gamma_{p,q}(\tau_{\mathrm{hit}}(p,q))) \rangle|}{\|  \gamma_{p,q}'(0) \|} \leq C_{1}'' \epsilon(h).
\ee

However, by Assumption \ref{DefLocallyWellBehaved}, we have
\be 
\sup_{\|x - q \| \leq \epsilon(h)} \| \eta(x) - \eta(s(q)) \| = O(\epsilon(h)).
\ee 

Combining this bound with Inequality \eqref{IneqAngleCondNoHatN42}, there exists $C_{2} > 0$ so that 

\be 
\frac{| \langle \gamma_{p,q}'(0), \eta(s(q)) \rangle |}{\|  \gamma_{p,q}'(0) \|} \leq C_{2} \epsilon(h)
\ee 

for all $(p,q) \in \mathcal{E}_{4,2} \backslash \mathcal{E}_{4,1}$. Using the bound $\| \gamma_{p,q}'(0) \| \leq h^{-1} \epsilon(h)$ for $(p,q) \in \mathcal{E}_{4,2} \backslash \mathcal{E}_{4,1}$,

\be  \label{IneqPNearOrthoNHat2}
| \langle \gamma_{p,q}'(0), \eta(s(q)) \rangle | \leq C_{2} h^{-1} \epsilon(h)^{2}
\ee

for all $(p,q) \in \mathcal{E}_{4,2} \backslash \mathcal{E}_{4,1}$; applying Assumption  \ref{DefLocallyWellBehaved} we can strengthen this to

\be 
\sup_{ q' \, : \, \|q' - q \| \leq 10 \epsilon(h)} | \langle \gamma_{p.q}'(0), \eta(s(q')) \rangle | \leq C_{3} h^{-1} \epsilon(h)^{2}
\ee 

for some $C_{3} > 0$. However, this also implies that $\gamma_{p,q}'(0)$ must have a very small velocity in the direction of $\partial S$, and so applying Lemma \ref{LemmaLinApprox} gives
\be  \label{IneqQVeryCloseNHat2}
\inf_{q'' \in \partial S} \| q - q'' \| = O(h^{-1} \epsilon(h)^{2}) + O(\epsilon(h)) = O(h^{-1} \epsilon(h)^{2}).
\ee

Putting together Inequalities \eqref{IneqPNearOrthoNHat2} and \eqref{IneqQVeryCloseNHat2}, there exists some $0 < C_{4} < \infty$ so that

\be 
| \langle \gamma_{p,q}'(0), \eta(s(q)) \rangle | \leq C_{4} h^{-1} \epsilon(h)^{2}, \, q \in (\partial S)_{C_{4} h^{-1} \epsilon(h)^{2}}
\ee 
for all $(p,q) \in \mathcal{E}_{4,2} \backslash \mathcal{E}_{4,1}$. This bound is identical in form to the displayed expression \eqref{EqMainCalcConcHatE4}, and so using the same calculations we conclude

\be 
\P[(P,Q) \in \mathcal{E}_{4,2} \backslash \mathcal{E}_{4,1}] = o(h).
\ee

Combining this with Inequality \eqref{IneqEasyEps41}, we conclude
\be 
\P[N_{\partial S}(P,Q,0,h) > 1] = o(h).
\ee 

Combining this with Inequality \eqref{IneqNoBadHats} completes the proof.

\end{proof}

Finally, combining Inequalities \eqref{IneqE1Small}, \eqref{IneqE2Small}, \eqref{IneqE3Small} and \eqref{IneqE4Small} we conclude 
\be 
\P[(P,Q) \in A, \, N_{\partial S}(P,Q,0,h) \neq \hat{N}_{\partial S}(P,Q,0,h)] = o(h).
\ee

Actually, these four inequalities imply the slightly stronger bound:

\be 
\P[(P,Q) \in A, \, N_{\partial S}(P,Q,0,h) \neq \mathbbm{1}_{\hat{N}_{\partial S}(P,Q,0,h) > 0}] = o(h).
\ee 

Thus, for all fixed $n \in \mathbb{N}$ 

\be \label{EqGoodHLimBoundComp}
\lim_{h \rightarrow 0} h^{-1} \E[ \min(n, &N_{\partial S}(P,Q,0,h) \mathbbm{1}_{(P,Q) \in A} )] \\
&= \lim_{h \rightarrow 0} h^{-1} \P[\{ \hat{N}_{\partial S}(P,Q,0,h) > 0\} \cap \{(P,Q) \in A \}].
\ee

Define
\be
f(n,h,p,q) = \min(n+1, N_{\partial S}(p,q,0,h) \mathbbm{1}_{(p,q) \in A} )
\ee
and recall the definition of $v^+$ from Equation \eqref{EqDefVPLate}. Applying the monotone convergence theorem and then the dominated convergence theorem to justify the two interchanges of limits, and then applying Assumption \ref{DefLocallyWellBehaved}, Equality \eqref{EqGoodHLimBoundComp} implies   
\be
\lim_{h \rightarrow 0} h^{-1} \E[N_{\partial S}(P,Q,0,h) \mathbbm{1}_{(P,Q) \in A} ] &=  \lim_{h \rightarrow 0} h^{-1} \E[ \lim_{n \rightarrow \infty} f(n,h,P,Q) ] \\
&= \lim_{h \rightarrow 0} h^{-1} \E[ \sum_{n=0}^{\infty} ( f(n+1,h,P,Q) - f(n,h,P,Q) ) ]  \\
&=  \lim_{h \rightarrow 0} \sum_{n=0}^{\infty}  h^{-1} \E[ f(n+1,h,P,Q) -  f(n,h,P,Q) ]  \\
&= \sum_{n=0}^{\infty}  \lim_{h \rightarrow 0} h^{-1} \E[ f(n+1,h,P,Q) -  f(n,h,P,Q) ]  \\
&\stackrel{{\scriptsize \textrm{Eq. }}\ref{EqGoodHLimBoundComp}}{=}   \lim_{h \rightarrow 0} h^{-1} \P[\{ \hat{N}_{\partial S}(P,Q,0,h) > 0\} \cap \{(P,Q) \in A \}] \\
&= 2 \lim_{h \rightarrow 0} h^{-1} \int_{A \cap (\partial S) } \int_{\mathbb{R}^d} \mu_{H}(p,q) \cdot h \langle v^+(p,q), \eta(q) \rangle \mathrm{d}p \mathrm{d}q\\
&= 2 \int_{A \cap (\partial S) } \int_{\mathbb{R}^d} \mu_{H}(p,q) \cdot \langle v^+(p,q), \eta(q) \rangle \mathrm{d}p \mathrm{d}q, 
\ee
where the formula in the third equality holds because $\hat{N}_{\partial S}(P,Q,0,h)$ counts the number of intersections of \emph{linear} trajectories.

Combining this with Equality \eqref{EqBreakUpPhiDeriv}, we conclude  that 
\be
\frac{d \Phi^{+}(t)}{dt} =  \sum_{A \in \mathfrak{R}} \int_{A \cap (\partial S) } \int_{\mathbb{R}^d} \mu_{H}(p,q) \cdot \langle v^+(p,q), \eta(q) \rangle \mathrm{d}p \mathrm{d}q,
\ee
so that
\be
\begin{split}
\Phi^+ =  \int_0^T \int_{\partial S} \int_{\mathbb{R}^d} \mu_{H}(p,q) \cdot \langle v^+(p,q), \eta(q) \rangle \mathrm{d}p \mathrm{d}q \mathrm{d}t \\
=T \cdot \int_{\partial S} \int_{\mathbb{R}^d}  \mu_{H}(p,q) \cdot \langle v^+(p,q), \eta(q) \rangle \mathrm{d}q.
\end{split}
\ee

\end{proof}

\subsection{Proof of Lemma \ref{LemmSuffLoc}} \label{SecLemmSuffLoc}

Since $\mathfrak{A} \subset \mathfrak{B}$, it is clear that $\inf_{S \in \mathfrak{A}} \Phi(K_{T},S) \geq \inf_{S \in \mathfrak{B}} \Phi(K_{T},S)$. Thus, it remains to show only that $\inf_{S \in \mathfrak{A}} \Phi(K_{T},S) \leq \inf_{S \in \mathfrak{B}} \Phi(K_{T},S)$. 

Fix some set $S \in \mathfrak{B}$ with $0 < \pi(S) < \frac{1}{2}$ and fix some constant $0 < \epsilon < \frac{\pi(S)}{100}$. Since $\pi$ has a continuous density with respect to Lebesgue measure, there exists a countable collection of open rectangles $\{ \hat{R}_{i}\}_{i \in \mathbb{N}}$ such that the following 
both hold:
\be 
S & \subset \cup_{i \in \mathbb{N}} \hat{R}_{i} \\
\pi(S) &\leq \sum_{i \in \mathbb{N}} \pi(\hat{R}_{i}) + \epsilon.
\ee 
The existence of a collection of rectangles with this property is shown during standard proofs of the Lebesgue regularity theorem (see \textit{e.g.} the proof of Lemma 1.2.12 of \cite{tao2011introduction}). Since $\pi$ is a probability measure with a continuous density with respect to Lebesgue measure, there exists some $n \in \mathbb{N}$ such that
\be 
\pi(S \Delta \cup_{i=1}^{n} \hat{R}_{i}) \leq 2 \epsilon,
\ee 
where $\Delta$ denotes the symmetric difference of two sets, namely $A \Delta B:=   (A \cup B) \backslash (A \cap B)$ for any two sets $A$ and $B$.
Let $R = \cup_{i=1}^{n} \hat{R}_{i}$. We next note that $R$ can be covered by a ``rounded" polyhedron that is an element of $\mathfrak{A}$, and which has only slightly larger measure. We give an explicit construction of such a covering here. 

To make the notation more legible, we denote by $A(\delta) = A_{\delta}$ the $\delta$-thickening of a set $A$. Since $R$ is a finite union of rectangles, there exists some $\delta_{1} > 0$ so that $R(\delta) \in \mathfrak{A}$ for all $0 < \delta < \delta_{1}$. Since $\pi$ is a probability measure with continuous density, there exists some $\delta_{2} > 0$ so that $\pi(R(\delta) \backslash R) \leq \epsilon$ for all $0 < \delta < \delta_{2}$. Defining $U = R(\frac{\min(\delta_{1},\delta_{2})}{2})$, we have 
\be 
U &\in \mathfrak{A} \\
\pi(S \Delta U) &\leq \pi(S \Delta R) + \pi(U \backslash R) \leq 3 \epsilon.
\ee 
For any $S \in \mathfrak{B}$ such that $12 \epsilon \leq \pi(S) \leq \frac{1}{2}$, we then have

\be \label{IneqIntRelSetCond}
\Phi(K_{T},U) &= \frac{\P[Q \in U, \gamma_{P,Q}(T) \notin U]}{\pi(U) } \\
&\leq \frac{\P[Q \in S, \gamma_{P,Q}(T) \notin S] + \P[Q \in U \backslash A] + \P[\gamma_{P,Q}(T) \in S \backslash U]}{\pi(U)}. \\
&\leq \frac{\P[Q \in S, \gamma_{P,Q}(T) \notin S] + \P[Q \in U \backslash S] + \P[\gamma_{P,Q}(T) \in S \backslash U]}{\pi(S) - \pi(S \backslash U)  } \\
&\leq \frac{\P[Q \in S, \gamma_{P,Q}(T) \notin S] + 6 \epsilon}{\pi(S) - 3 \epsilon } \\
&\leq  \frac{\P[Q \in S, \gamma_{P,Q}(T) \notin S] + 6 \epsilon}{\pi(S) } \left(1 + \frac{6 \epsilon}{\pi(S)} \right) \\
&= \left( \Phi(K_{T},S) + \frac{6 \epsilon}{\pi(S) } \right)\left(1 + \frac{6 \epsilon}{\pi(S)} \right),
\ee 
where the fourth inequality holds since $\pi(S) \geq 12 \epsilon$ and $\epsilon < \frac{1}{100}$. Since $\epsilon > 0$ can be chosen to be arbitrarily small, this implies 
\be 
\inf_{U \in \mathfrak{A}} \Phi(K_{T},U) \leq \Phi(K_{T},S).
\ee 
for any $S \in \mathfrak{B}$  such that $0<\pi(S) <\frac{1}{2}$. Since $S \in \mathfrak{B}$ is arbitrary, this completes the proof.

\section{Example Calculations}

We give some generic bounds related to metastability, and then analyze our main examples. We recall some general definition that will be used throughout this appendix.

\begin{defn} [Trace Chain]
Let  $K$ be the transition kernel of an ergodic Markov chain on state space $\Omega$ with stationary measure $\mu$, and let $S \subset \Omega$ be a subset with $\mu(S) > 0$. Let $\{X_{t}\}_{t \geq 0}$ be a Markov chain evolving according to $K$, and iteratively define
\be 
c_{0} &= \inf \{t \geq 0 \, : \, X_{t} \in S \} \\
c_{i+1} &= \inf \{t > c_{i} \, : \, X_{t} \in S \}.
\ee 
Then 
\be \label{EqTraceCoup}
\hat{X}_{t} = X_{c_{t}}, \quad t \geq 0
\ee 
is the \textit{trace of $\{X_{t}\}_{t \geq 0}$ on $S$}. Note that $\{\hat{X}_{t}\}_{t \geq 0}$ is a Markov chain with state space $S$, and so this procedure also defines a transition kernel with state space $S$. We call this kernel the \textit{trace of the kernel $K$ on $S$}.
\end{defn}

\begin{defn} [Hitting Time]
Let $\{X_{t}\}_{t \geq 0}$ be a Markov chain with initial point $X_{0} = x$ and let $S$ be a measurable set. Then 
\be 
\tau_{x,S} = \inf \{t \geq 0 \, : \, X_{t} \in S\}
\ee 
is called the \textit{hitting time} of $S$.
\end{defn}

\subsection{Generic Metastability Bounds} \label{SecGenMetaBd}

We recall some results from our companion paper \cite{mangoubi2018simple}. We need some standard definitions:

\begin{defn} [Trace Chain]
Let  $K$ be the transition kernel of an ergodic Markov chain on state space $\Omega$ with stationary measure $\mu$, and let $S \subset \Omega$ be a subset with $\mu(S) > 0$. Let $\{X_{t}\}_{t \geq 0}$ be a Markov chain evolving according to $K$, and iteratively define
\be 
c_{0} &= \inf \{t \geq 0 \, : \, X_{t} \in S \} \\
c_{i+1} &= \inf \{t > c_{i} \, : \, X_{t} \in S \}.
\ee 
Then 
\be \label{EqTraceCoup}
\hat{X}_{t} = X_{c_{t}}, \quad t \geq 0
\ee 
is the \textit{trace of $\{X_{t}\}_{t \geq 0}$ on $S$}. Note that $\{\hat{X}_{t}\}_{t \geq 0}$ is a Markov chain with state space $S$, and so this procedure also defines a transition kernel with state space $S$. We call this kernel the \textit{trace of the kernel $K$ on $S$}.
\end{defn}

\begin{defn} [Hitting Time]
Let $\{X_{t}\}_{t \geq 0}$ be a Markov chain with initial point $X_{0} = x$ and let $S$ be a measurable set. Then 
\be 
\tau_{x,S} = \inf \{t \geq 0 \, : \, X_{t} \in S\}
\ee 
is called the \textit{hitting time} of $S$.
\end{defn}

Denote by $\{Q_{\beta}\}_{\beta \geq 0}$ the transition kernels of ergodic Markov chains with stationary measures $\{ \mu_{\beta}\}_{ \beta \geq 0}$ on common state space $\Omega$, which we take to be a convex subset of Euclidean spac. Throughout, we will always use the subscript $\beta$ to indicate which chain is being used - for example, $\Phi_{\beta}(S)$ is the conductance of the set $S$ with respect to the chain $Q_{\beta}$. For any set $S$ with $\pi_{\beta}(S) > 0$, define the restriction $\pi_{\beta} |_{S}$ of $\pi_{\beta}$ to $S$
\be 
\pi_{\beta} |_{S}(A) = \frac{\pi_{\beta}(S \cap A)}{\pi_{\beta}(S)}.
\ee

Fix $S \subset \Omega$ with $\inf_{\beta \geq 0} \pi_{\beta}(S) \equiv c_{1} > 0$ and, for all $\beta \geq 0$. Let $\gd_{\beta}, \bd_{\beta}, \cov_{\beta} \subset S$ satisfy 
\be 
\gd_{\beta} \subset \cov_{\beta}, \qquad \bd_{\beta} \subset \cov_{\beta}^{c}.
\ee 
In the following assumption, we think of the set $\gd_{\beta}$ as the points that are ``deep within" the mode $S$, the points $\bd_{\beta}$ as the points that are ``far in the tails" of the target distribution, and the ``covering set" $\cov_{\beta}$ as a way of separating these two regions.

\begin{assumptions} \label{AssumptionsMeta1}
 We assume the following all hold for $\beta > \beta_{0}$ sufficiently large:
\begin{enumerate}
\item \textbf{Small Conductance:} There exists some $c > 0$ such that $\Phi_{\beta}(S) \leq e^{-c \beta}$.
\item \textbf{Rapid Mixing Within $\gd_{\beta}$:}  Let $\hat{Q}_{\beta}$ be the Metropolis-Hastings chain with proposal kernel $Q_{\beta}$ and target distribution $\pi_{\beta} |_{S}$. There exists some function $r_{1}$ bounded by a polynomial such that 
\be  \label{IneqMeta1MixingRestrictions}
\sup_{x \in \gd_{\beta}} \| \hat{Q}_{\beta}^{r_{1}(\beta)}(x,\cdot) - \pi_{\beta} |_{S}(\cdot) \|_{\mathrm{TV}} \leq \beta^{-2} \Phi_{\beta}(S).
\ee 
\item \textbf{Never Stuck In $\cov_{\beta}\backslash \gd_{\beta}$:} There exists some function $r_{2}$ bounded by a polynomial such that 
\be \label{IneqNeverStuckInEscM1}
\sup_{x \in \cov_{\beta}\backslash \gd_{\beta}} \P[\tau_{x, \gd_{\beta} \cup S^{c}} > r_{2}(\beta)] \leq  \beta^{-2} \Phi_{\beta}(S).
\ee 
\item \textbf{Never Hitting $\cov_{\beta}^{c}$:} We have 
\be \label{IneqNeverHittingBad1}
\sup_{x \in \gd_{\beta} } \P[\tau_{x, \cov_{\beta}^{c}} < \min( r_{1}(\beta) + r_{2}(\beta) + 1, \tau_{x,S^{c}})] \leq \Phi_{\beta}(S)^{4}.
\ee 
\end{enumerate}
\end{assumptions}

Theorem 1 of \cite{mangoubi2018simple} is:

\begin{lemma} [Hitting Times and Conductance] \label{LemmaMeta1}
Let Assumptions \ref{AssumptionsMeta1} hold, and fix a point $x$ that is in  $\gd_{\beta} $ for all $\beta > \beta_{0}(x)$ sufficiently large. Then for all $\epsilon > 0$,
\be 
\P \left[ \frac{\log(\tau_{x,S^{c}})}{\log(\Phi_{\beta}(S))} > 1 + \epsilon \right] = o(1).
\ee 

\end{lemma}

The other result we need has assumptions:

\begin{assumptions} \label{AssumptionsMeta2}
Let $\Omega = S^{(1)} \sqcup S^{(2)}$ be a partition of $\Omega$ into two pieces. Set $\Phi_{\min} = \min(\Phi_{\beta}(S^{(1)}), \Phi_{\beta}(S^{(2)}))$ and $\Phi_{\max} = \max(\Phi_{\beta}(S^{(1)}), \Phi_{\beta}(S^{(2)}))$. Assume that:

\begin{enumerate}
\item \textbf{Metastability of Sets:} Each set $S^{(i)}$ satisfies Assumption \ref{AssumptionsMeta1} (with $\Phi_{\beta}(S)$ replaced by $\Phi_{\max}$ in Part \textbf{(1)} and replaced by $\Phi_{\min}$ in Parts \textbf{(2-4)}). We use the superscript $(i)$ to extend the notation of that assumption in the obvious way.
\item \textbf{Lyapunov Control of Tails:} 
Denote by $B_{r}(x)$ the ball of radius $r > 0$ around a point $x \in \Omega$. Assume there exist $0 < m, M < \infty$ satisfying
\be \label{IneqLyapContain}
\cup_{i=1}^{2} \cov_{\beta}^{(i)} \subset B_{M}(0), \qquad
B_{m}(0) \subset \cup_{i=1}^{2} \gd_{\beta}^{(i)}.
\ee 
Assume there exist a collection of privileged points $s_{i} \in \gd_{\beta}^{(i)}$ such that the function $V_{\beta}(x) = e^{\beta \min_{1 \leq i \leq k} \| x - s_{i} \|}$ satisfies
\be \label{IneqLyapMain}
(Q_{\beta} V_{\beta})( x) \leq (1 - \frac{1}{r_{3}(\beta)}) V_{\beta}( x) + r_{4} e^{\ell \beta}
\ee 
for all $x \in \Omega$, where $r_{3}, r_{4}$ are bounded by polynomials and $0 \leq \ell < m$. 
\item \textbf{Never Hitting $\cov_{\beta}^{c}$:} We have the following variant of Inequality \eqref{IneqNeverHittingBad1}:
\be \label{IneqNeverHittingBad2}
\sup_{x \in \gd_{\beta}^{(1)} \cup \gd_{\beta}^{(2)} } \P[\tau_{x, (\cov_{\beta}^{(1)} \cup \cov_{\beta}^{(2)})^{c}} < \Phi_{\min}^{-2}] &\leq \Phi_{\min}^{4}.
\ee 

\item \textbf{Non-Periodicity:} We have 
\be \label{IneqNP1}
\inf_{\beta} \inf_{x \in S^{(1)} } Q_{\beta}(x,S^{(1)}) \equiv c_{2}^{(1)} > 0, \quad 
\inf_{\beta}  \inf_{x \in S^{(2)} }  Q_{\beta}(x,S^{(2)}) \equiv c_{2}^{(2)} > 0, \\
\ee 
and,
\be \label{IneqNP2}
\sup_{x \in S^{(1)}} Q_{\beta}(x,S^{(2)} \backslash \gd_{\beta}^{(2)}) < \Phi_{\min}^{4}, \quad
\sup_{x \in S^{(2)}} Q_{\beta}(x,S^{(1)} \backslash \gd_{\beta}^{(1)}) < \Phi_{\min}^{4}.
\ee 

\end{enumerate} 
\end{assumptions}

Theorem 2 of \cite{mangoubi2018simple} gives:

\begin{lemma} [Spectral Gap and Conductance] \label{LemmaMeta2}
Let Assumptions \ref{AssumptionsMeta2} hold. Denote by $\lambda_{\beta}$ and $\Phi_{\beta}$ the spectral gap and conductance of $Q_{\beta}$. Then
\be \label{EqSpecGapCondConc}
\lim_{\beta \rightarrow \infty} \frac{\log(\lambda_{\beta})}{\log(\Phi_{\min})} = \lim_{\beta \rightarrow \infty} \frac{\log(\Phi_{\beta})}{\log(\Phi_{\min})} = 1.
\ee 

\end{lemma}

\subsection{Proof of Theorem \ref{ThmHmcMultimodal}} \label{AppSubsecPfHmcMulti}

We prove Theorem \ref{ThmHmcMultimodal}, keeping the notation of that theorem. We begin by proving Equation \eqref{EqMultiCheegAsym}. By Theorem \ref{thm:1}, 
\be \label{IneqCondUpperBoundThm1}
\Phi_{\sigma} &= \Phi^+ \cdot \mathbb{E}_\mathbb{Q}\bigg[\frac{1}{N_{\{0\}}} \cdot \mathbbm{1}\{N_{\{0\}} \, \mathrm{odd}\}\bigg] \bigg/ \pi_{\sigma}((-\infty,0)) \\
&= T_{\sigma} \cdot \int_{\{0\}} f_{\sigma}(q) \mathrm{d}q \cdot  \mathbb{E}_\mathbb{Q}\bigg[\frac{1}{N_{\{0\}}} \cdot \mathbbm{1}\{N_{\{ 0\}} \, \mathrm{odd}\}\bigg] \\
&= \sigma \frac{1}{\sqrt{2 \pi} \sigma} e^{-\frac{1}{2 \sigma^{2}}}  \mathbb{E}_\mathbb{Q}\bigg[\frac{1}{N_{\{0\}}} \cdot \mathbbm{1}\{N_{\{ 0\}} \, \mathrm{odd}\}\bigg]   \\
&\leq   \frac{1}{\sqrt{2 \pi}} e^{-\frac{1}{2 \sigma^{2}}}.
\ee

Taking logs, this bound immediately implies the upper bound in Equality \eqref{EqMultiCheegAsym}. The lower bound required for Equality \eqref{EqMultiCheegAsym} is very weak, and we give rather weak estimates to obtain it. Define $I_{\sigma} = (-\sigma^{20},0)$ and $ J_{\sigma} = (\sigma^{12},\sigma^{8})$. Noting that $\sup_{|x| \leq \sigma^{7}} | \nabla \log(f_{\sigma}(x)) | = O(\sigma)$, we have by Lemma \ref{LemmaLinApprox} that

\be 
\inf_{q \in I_{\sigma}, \, p \in J_{\sigma}} \gamma_{p,q}(T_{\sigma}) > 0
\ee 
for all $\sigma$ sufficiently small. Thus,  
\be 
\Phi_{\sigma} &\geq \P[Q \in I_{\sigma}, \, P \in J_{\sigma}] \\
&= \Omega(e^{-\frac{1}{2 \sigma^{2}}} \times \sigma^{32}).
\ee 
This immediately proves the lower bound on Equality \eqref{EqMultiCheegAsym}.

Next, we prove Equation \eqref{EqMultiCheegHitting}, which is the bulk of the proof of Theorem \ref{ThmHmcMultimodal}. The proof will entirely consist of verifying the assumptions \ref{AssumptionsMeta1}, \ref{AssumptionsMeta2} of Lemmas \ref{LemmaMeta1} and \ref{LemmaMeta2}.  In order to keep the proof to a reasonable length, some bounds in this proof are not very explicit. Throughout the proof, we generally absorb all terms that are of smaller order (with respect to the parameter $\sigma$) from line to line in our calculations in order to minimize notational clutter. In particular, we freely use some very weak bounds without comment - most frequently, the fact that a Gaussian with mean 0 and small variance $\sigma^{2} \ll 1$ has negligible probability of taking values as large as \textit{e.g.} $\sigma^{-10}$ on the time-scale of our analysis. Our argument will involve coupling several Markov chains, and so the standard notation for Markov chains ($\{X_{t}\}_{t \in \mathbb{N}}$, $\{Y_{t}\}_{t \in \mathbb{N}}$, etc) will be recycled during the proof; this notational re-use is clearly indicated when it occurs.

We will also find it useful to consider the following auxillary transition kernels. First, denote by $\hat{K}_{\sigma}$ the Metropolis-Hastings  transition kernel on $(-\infty,0)$ that has as its proposal kernel $K_{\sigma}$ and as its target distribution the density 
\be  
\hat{\rho}_{\sigma}(x) = 2 f_{\sigma}(x)  
\ee
on $(-\infty,0)$. In other words, $\hat{K}_{\sigma}$ evolves like $K_{\sigma}$, but rejects moves that take the Markov chain outside of $(-\infty,0)$. 

Next, let $\ell(x)= - \log(f_{\sigma}(x))$, fix $0 < \delta < \frac{1}{20}$, and define $f(x)$ to be the unique quadratic satisfying 
\be \label{UniqueQuadratic1}
f(-2 \delta) = \ell(-2 \delta), \, f'(-2 \delta) = \ell'(-2 \delta), \, f''(-2 \delta) = \ell''(-2 \delta),
\ee 
and define the potential function $U_{\sigma,\delta}$ by the formula 
\be \label{UniqueQuadratic2}
U_{\sigma,\delta}(x) &= \begin{cases}
 - \log(f_{\sigma}(x)), \qquad x < - 2 \delta \\
f(x), \qquad \qquad \quad \, \, x \geq -2 \delta.
\end{cases}
\ee

Define the density $\mu_{\sigma,\delta}(x) \propto e^{-U_{\sigma,\delta}(x)}$ and let $Q_{\sigma,\delta}$ be the transition kernel defined by Algorithm \ref{alg:isotropic_HMC}, with target distribution $\mu_{\sigma,\delta}$ and integration time $T_{\sigma}$.

Throughout the proof, we will construct couplings using the  ``forward mapping" representations of $K_{\sigma}$ and $Q_{\sigma,\delta}$.  Inspecting Algorithm \ref{alg:isotropic_HMC}, we see that the output of the algorithm is entirely defined by deterministic parameters and the i.i.d. sequence of momentums $\{ p_{t}\}_{t \in \mathbb{N}}$ sampled in Step 3 of the for loop. In other words, Algorithm \ref{alg:isotropic_HMC} gives a \textit{forward mapping representation} of the transition kernels $K_{\sigma}$ and $Q_{\sigma,\delta}$. In particular, to couple two Markov chains $\{X_{t}\}_{t \in \mathbb{N}}$, $\{Y_{t}\}_{t \in \mathbb{N}}$ defined by Algorithm \ref{alg:isotropic_HMC}, it is enough to couple the associated i.i.d. sequences. All of our couplings will be defined this way; if two chains share the \textit{same} sequence, we call this the \textit{identity coupling}.

In the notation of  Lemmas \ref{LemmaMeta1} and \ref{LemmaMeta2}, we will use the partition
\be 
S^{(1)} = (-\infty,0), \qquad S^{(2)} = [0,\infty),
\ee 
with decomposition of $S^{(1)}$
\be 
\gd_{\beta}^{(1)} = (-\sigma^{-9},0), \quad \cov_{\beta}^{(1)} = (-\sigma^{-10},0), \quad \bd_{\beta}^{(1)} = (-\infty, -\sigma^{-10})
\ee 
and $\gd_{\beta}^{(2)}$, $\cov_{\beta}^{(2)}$, $\bd_{\beta}^{(2)}$ defined analogously (see Figure \ref{FigPartitionEx}). Note that Part \textbf{(1)} of Assumptions \ref{AssumptionsMeta1} follows immediately from Inequality \eqref{EqMultiCheegAsym}, which we have already proved. 

\begin{figure}[H]
\includegraphics[scale=0.5]{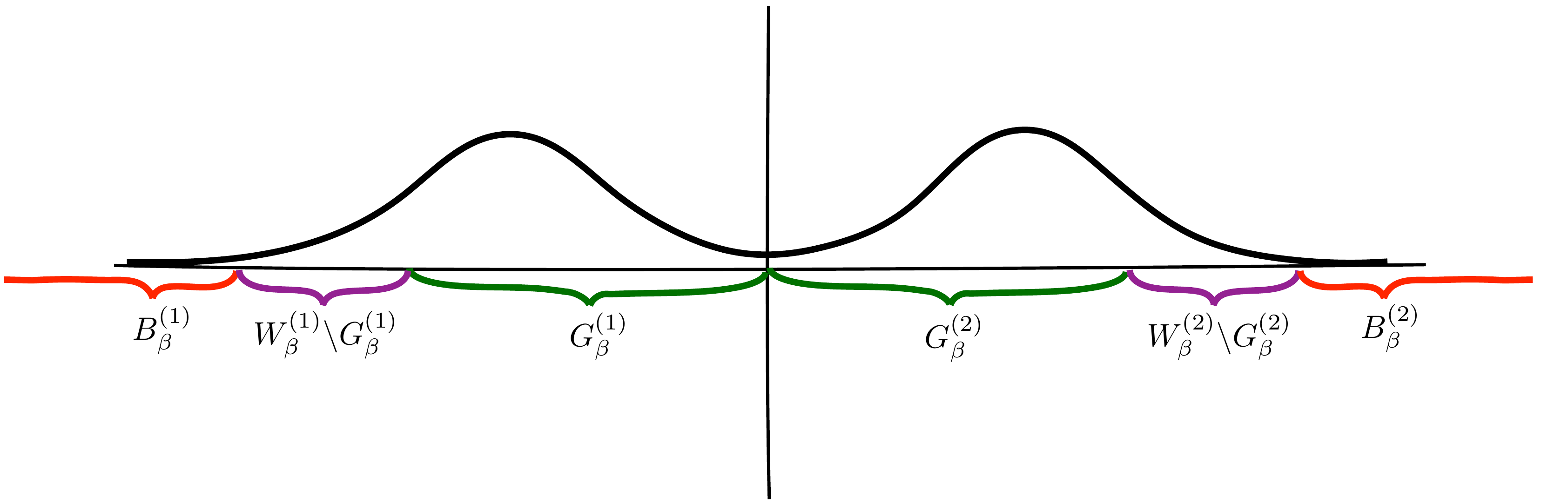}
\caption{A cartoon plot of the target density $\mu_{\sigma}$ with the regions illustrated. Note that we have substantially compressed the regions; in a to-scale drawing, $\bd_{\beta}$ would not be visible. \label{FigPartitionEx}}
\end{figure}

We note that Part \textbf{(1)} of Assumptions \ref{AssumptionsMeta1} follows immediately from Inequality \eqref{IneqCondUpperBoundThm1}. We skip Part \textbf{(2)} for now to prove some Lyapunov-like bounds for $K_{\sigma}$ and $\hat{K}_{\sigma}$:

\begin{lemma} \label{LemmaLyapunovEx}
Let $V_{\sigma}(x) = e^{\sigma^{-1} \min(\|x - 1\|, \|x+1\|)}$. Then there exist $0 < \alpha \leq 1$, $0 \leq M < \infty$ and $C = C(\sigma)$ bounded by a polynomial so that for all $K \in \{K_{\sigma}, \hat{K}_{\sigma}\}$ and $x \in (-\infty,-M \sigma)$, 
\be \label{IneqLyapMulti}
(K V_{\sigma})(x) \leq (1 - \alpha) V_{\sigma}(x) + C.
\ee 
Furthermore, Part \textbf{(2)} of Assumptions \ref{AssumptionsMeta2} holds.
\end{lemma}

\begin{proof}

After a few small changes to the first lines, the proof is very similar to the proof of Lemma 4.3 in our companion paper \cite{mangoubi2018simple}. Observe that
\be 
\inf_{x \in (-\infty,-10 \sigma)} \frac{d^{2}}{dx^{2}} -\log(f_{\sigma}(x)) \geq \frac{1}{8 \sigma^{2}}.
\ee 
In particular, $f_{\sigma}$ is strongly log-concave on the interval $(-\infty,-10 \sigma)$ with parameter at least $\frac{1}{8 \sigma^{2}}$. Fix $M > 10$, $x \in (-\infty,-M \sigma)$, $K \in \{K_{\sigma}, \hat{K}_{\sigma}\}$ and let $X \sim K_{\sigma}(x,\cdot)$, and let $P \sim \mathcal{N}(0,1)$ be the associated momentum variable for the random mapping representation of $K_{\sigma}$; by Theorem 7 of \cite{mangoubi2017concave}, we have
\be \label{LyapProofPt1}
(K V_{\sigma})(x) &\leq (1 -\alpha) V_{\sigma}(x) + C_{\sigma} + \E[V_{\sigma}(X) \textbf{1}_{X > -10 \sigma}] \\ 
&= (1 -\alpha) V_{\sigma}(x) + C_{\sigma} + \E[V_{\sigma}(X) \textbf{1}_{\sup_{0 \leq s \leq \sigma} \gamma_{P,x}(s) >  -10 \sigma}].
\ee 

The proof of Inequality \eqref{IneqLyapMulti} now concludes as the proof of Lemma 4.3 in \cite{mangoubi2018simple} following Inequality (4.7) (which is nearly identical to Inequality \eqref{LyapProofPt1}).

Finally, Part \textbf{(2)} of Assumptions \ref{AssumptionsMeta2} immediately follows from Inequality \eqref{IneqLyapMulti} in the case $K = K_{\sigma}$ and the trivial inequality
\be 
\sup_{|x| \leq M \sigma} (K_{\sigma} V_{\sigma})(x) \leq e^{\sigma^{-3}}
\ee 
for any fixed $M$ and all $\sigma < \sigma_{0} = \sigma_{0}(A)$ sufficiently small.
\end{proof}

Most of our proof consists of checking that Part \textbf{(2)} of Assumptions \ref{AssumptionsMeta1} holds: 

\begin{lemma} \label{ThmMulti1AppMainMixingMode}
With notation as above, Part \textbf{(2)} of  Assumptions \ref{AssumptionsMeta1} is satisfied.
\end{lemma}

\begin{proof}

We begin by proving an initial weak estimate of the mixing of both $K_{\sigma}$ and $\hat{K}_{\sigma}$ on the mode $(-\infty,-2 \delta)$:

\begin{prop} \label{LemmaInitialModeMixingEstimate}
With notation as above, there exist some constants $0 < c_{1}, c_{2},C_{1}, C_{2} < \infty$ so that
\be   \label{IneqMixingBulkNoHat}
\sup_{-\sigma^{-10} < x,y < -2\delta} \| K_{\sigma}^{S}(x,\cdot) -  K_{\sigma}^{S}(y,\cdot) \|_{\mathrm{TV}} \leq 2 \,C_{2} e^{-c_{2} \sigma^{-1}}
\ee
and
\be  \label{IneqMixingBulkHat}
\sup_{-\sigma^{-10} < x,y < -2\delta} \| \hat{K}_{\sigma}^{S}(x,\cdot) -  \hat{K}_{\sigma}^{S}(y,\cdot) \|_{\mathrm{TV}} \leq 2 \,C_{2} e^{-c_{2} \sigma^{-1}}
\ee
for $S > C_{1} \sigma^{-c_{1}}$.
\end{prop}

\begin{proof}
Our argument is obtained with the coupling method. Let $\{X_{t}\}_{t \in \mathbb{N}}$ be a Markov chain with transition kernel $K_{\sigma}$ and starting point $X_{1} = x$. Let $\{Y_{t}\}_{t \in \mathbb{N}}$ be a Markov chain with transition kernel $Q_{\sigma,\delta}$ and starting point $Y_{1} = x$. Finally, let $\{Z_{t}\}_{t \in \mathbb{N}}$ with transition kernel $Q_{\sigma,\delta}$, started at stationarity with $Z_{1} \sim \mu_{\sigma,\delta}$. 

We now define a coupling of $\{X_{t},Y_{t},Z_{t}\}_{t \in \mathbb{N}}$. We begin by coupling $\{X_{t}, Y_{t}\}_{t \in \mathbb{N}}$ according to the \textit{identity coupling} given near the start of this section. Under this coupling, define the random times
\be 
\eta_{1} &= \inf \{ t \in \mathbb{N} \, : \, X_{t} > -\delta \} \\
\eta_{2} &= \inf \{ t \in \mathbb{N} \, : \, \frac{1}{2} p_{t}^{2} - \log(f_{\sigma}(X_{t})) > f_{\sigma}(- \frac{ \delta}{2}) \} \\
\eta_{3} &= \inf \{ t \in \mathbb{N} \, : \, X_{t} < -\sigma^{-10} \} \\
\eta &= \min(\eta_{1},\eta_{2}, \eta_{3}).
\ee 

We observe that, under this coupling, $X_{t} = Y_{t}$ for all $t \leq \eta$. Next, we extend this coupling of $\{X_{t},Y_{t}\}_{t \in \mathbb{N}}$ to a coupling of $\{X_{t},Y_{t},Z_{t}\}_{t \in \mathbb{N}}$. 

By Theorem 1 of \cite{mangoubi2017concave}, there exist $c_{1} = c_{1}(\delta) > 0$, $C_{1} = C_{1}(\delta) > 0$ so that 
\be  \label{IneqMixingModeCitePrev}
\| \mathcal{L}(Y_{T}) - \mathcal{L}(Z_{T}) \|_{\mathrm{TV}} \leq e^{-  \sigma^{-4}}
\ee 
for all $T > C_{1} \sigma^{-c_{1}}$. Thus, for fixed $T > C_{1} \sigma^{-c_{1}}$, it is possible to couple $\{Y_{t}\}_{t=1}^{T}$, $\{Z_{t}\}_{t=1}^{T}$ so that 
\be \label{IneqCoupVStrong}
\P[Y_{T} = Z_{T}] \geq 1 - e^{-  \sigma^{-4}}.
\ee 
We fix $T = \lceil C_{1} \sigma^{-c_{1}} \rceil$ and couple $\{Z_{t}\}_{t=1}^{T}$ to $\{Y_{t}\}_{t =1}^{T}$ in this way; for $t > T$ we allow $\{Z_{t}\}_{t \geq T+1}$, $\{Y_{t}\}_{t \geq T+1}$ to evolve independently. 

By Equation \eqref{IneqCoupVStrong}, we have under this coupling that 
\be \label{IneqSimpCoupTail1}
\P[\{X_{T} = Z_{T}\} \cap \{X_{T} \in (-\sigma^{-10}, - 2 \delta)\}] &\geq \P[X_{T} = Z_{t}] - \P[X_{T} \notin (-\sigma^{-10}, - 2 \delta)] \\
&\geq \P[X_{T} = Z_{t}] - \P[\eta < T] \\
&\geq  1  - e^{-\sigma^{-4}} - \P[\eta < T].
\ee
By the drift condition \eqref{IneqLyapMulti} and Markov's inequality, there exist constants $c_{2}' = c_{2}'(\delta) > 0$,  and $C_{2}' = C_{2}'(\delta) > 0$ so that 
\be \label{IneqSimpCoupTail2}
\P[ \eta_{1} < T] \leq C_{2}' e^{-c_{2}' \sigma^{-1}}.
\ee
By standard Gaussian tail inequalities, there exist constants $c_{2}'' = c_{2}''(\delta) > 0$,  and $C_{2}'' = C_{2}''(\delta) > 0$ so that 
\be \label{IneqSimpCoupTail3}
\P[ \eta_{2} < \min(\eta_{1},T)] \leq C_{2}'' e^{-c_{2}'' \sigma^{-2}}.
\ee

Thus, combining Inequalities \eqref{IneqSimpCoupTail1}, \eqref{IneqSimpCoupTail2} and \eqref{IneqSimpCoupTail3}, there exist constants $c_{2} = c_{2}(\delta) > 0$,  and $C_{2} = C_{2}(\delta) > 0$ so that
\be 
\P[\{X_{T} = Z_{T}\} \cap \{X_{T} \in (-\sigma^{-10}, - 2 \delta)] \geq 1 - C_{2} e^{-c_{2} \sigma^{-1}}.
\ee

This completes the proof of Inequality \eqref{IneqMixingBulkNoHat}. Inequality \eqref{IneqMixingBulkHat} follows from the same argument.
\end{proof}

Note that, although the total variation distance in Inequalities \eqref{IneqMixingBulkNoHat}, \eqref{IneqMixingBulkHat} is very small, it is still much larger than $\Phi_{\sigma}$ and so Proposition \ref{LemmaInitialModeMixingEstimate} does not imply Lemma \ref{ThmMulti1AppMainMixingMode}.  By iteratively applying this bound and paying more attention to small error terms, we find the improved estimate:

\begin{prop} \label{LemmaIterateHatMixing}
There exist constants $0 < c_{6}, c_{7}, C_{6}, C_{7} < \infty$ so that 
\be  \label{EqVeryStrongHatMixing2}
\sup_{-\sigma^{-9} < x,y < 0} \| \hat{K}_{\sigma}^{S}(x,\cdot) -  \hat{K}_{\sigma}^{S}(y,\cdot) \|_{\mathrm{TV}} \leq C_{6} e^{-c_{6} \sigma^{-5}}
\ee
for $S = \lceil C_{7} \sigma^{-c_{7}} \rceil$. 
\end{prop}

\begin{proof}
We initially fix $0 < \delta < \frac{1}{20}$ and consider pairs of points $-\sigma^{-0} < x,y < -2\delta$. To improve on the bound in Proposition \ref{LemmaInitialModeMixingEstimate}, we must control what can occur when coupling does not happen quickly. There are two possibilities to control: the possibility that $X_{t}$ goes above $-2 \delta$, and the possibility that $X_{t}$ goes below $-\sigma^{-10}$. The latter is easier to control; by the Lyapunov Inequality \eqref{IneqLyapMulti} and Markov's inequality, for all $\delta > 0$ there exist constants $c_{3} = c_{3}(\delta), C_{3} = C_{3}(\delta) > 0$ so that  
\be  \label{IneqVeryLowExcursionsVeryUnlikely}
\sup_{|X_{1}| \leq \sigma^{-\beta}} \P[\min_{1 \leq t \leq e^{\sigma^{-\beta}}} X_{t} < -\sigma^{-\beta - \delta}] &\leq e^{\sigma^{-\beta}} \, \sup_{|X_{1}| \leq \sigma^{-\beta}} \sup_{1 \leq t \leq e^{\sigma^{-\beta}}} \frac{\E[e^{|X_{t}|}]}{e^{\sigma^{-\beta - \delta}}} \\
&\leq e^{\sigma^{-\beta}} \left( \frac{e^{\sigma^{-\beta}} + \alpha^{-1} C_{\sigma}'}{e^{\sigma^{-\beta - \delta}}} \right) \\
&\leq C_{3} e^{-c_{3} \sigma^{-\beta - \delta}}
\ee
uniformly in $\beta \geq 1$.

The possibility that $X_{t}$ goes above $-2 \delta$ cannot be controlled in the same way, because it does not have negligible probability on the time scale of interest. Instead, we use the fact that $X_{t}$ will generally exit the interval $(-2 \delta, 0)$ fairly quickly, and that it often returns to the interval $(-\infty, - 2 \delta)$ when it does so. 

We now prove that the HMC Markov chain exits $(-2\delta, 0)$ quickly. To do so, we reset some of our notation in order to define a new coupling. We let $\{p_{t}\}_{t \geq 0}$ be a sequence of i.i.d. $\mathcal{N}(0,\sigma^{2})$ random variables. We let $\{X_{t}\}_{t \in \mathbb{N}}$ now denote a Markov chain evolving according to the kernel $K_{\sigma}$ defined in Algorithm \ref{alg:isotropic_HMC}, with this choice of update variable $\{p_{t}\}_{t \geq 0}$ in Step 3 of the algorithm, and some starting point $X_{1} = x \in (-2\delta,0)$. We let $\{ \hat{X}_{t}\}_{t \geq 0}$ be a Markov chain with transition kernel $\hat{K}_{\sigma}$  and coupled to $\{X_{t}\}_{t \geq 0}$ by the identity coupling. Finally, redefine  $\{Y_{t}\}_{t \in \mathbb{N}}$ to be the Markov chain: 
\be 
Y_{t} = X_{0} + \sum_{s=1}^{t-1} p_{s}.
\ee

Let $(\xi_{t}^{(1)}(s), \xi_{t}^{(2)}(s))_{s \geq 0}$ be solutions to Hamilton's equations \eqref{EqHamEq} with starting points $\xi_{t}^{(1)}(0) = p_{t}, \xi_{t}^{(2)}(0) = X_{t}$. Then define the stopping times:

\be 
\tau^{(1)} &= \min \{t  \in \mathbb{N} \, : \, \{\xi_{t}^{(2)}(s)\}_{s=0}^{\sigma} \cap (-2 \delta,0)^{c} \neq \emptyset \} \\
\tilde{\tau}^{(1)} &= \min \{t \in \mathbb{N} \, : \, X_{t} \notin (-2 \delta, 0) \}, \, \tau^{(2)} = \min \{t \in \mathbb{N} \, : \, Y_{t} \notin (-2 \delta, 0) \}.
\ee 
Roughly speaking, these are the first time that any of the Hamiltonian paths followed in the construction of $X_{t}$, or the first time that $X_{t}$, $Y_{t}$ themselves, exits $(-2 \delta,0)$. Note that we always have $\tau^{(1)} \leq \tilde{\tau}^{(1)}$.

Under our coupling, we have that $X_{t}, \, \hat{X}_{t} \leq Y_{t}$ for all $0 \leq t \leq \tau^{(1)}$. This immediately implies that 
\be 
\tau^{(2)} \leq \tilde{\tau}^{(1)}.
\ee  
By a straightforward calculation for the simple random walk $\{Y_{t}\}_{t \in \mathbb{N}}$, this implies the existence of constants $c_{4}' = c_{4}'(\delta) > 0$, $C_{4}' = C_{4}'(\delta) > 0$ so that 
\be \label{IneqLeavingIntermodalIneq}
\P[ \tilde{\tau}^{(1)} < C_{4}' \sigma^{-c_{4}'}] \geq \P[ \tau^{(2)} < C_{4}' \sigma^{-c_{4}'}] \geq \frac{999}{1000}.
\ee
By a direct calculation, for all $\epsilon > 0$,
\be 
\sup_{X_{0} = x \in (-2 \delta, 0)} \P[\sup_{0 \leq s \leq \sigma} \xi_{1}^{(2)}(s) > X_{0} + \epsilon] = e^{-\Omega(\sigma^{-2})}.
\ee 
Combining this observation with an application of Inequality \eqref{IneqLeavingIntermodalIneq} for a slightly larger value of $\delta$, we conclude that there exist constants $c_{4}'' = c_{4}''(\delta) > 0$, $C_{4}'' = C_{4}''(\delta) > 0$ so that 
\be \label{IneqLeavingIntermodalIneq2}
\P[ \tau^{(1)} < C_{4}'' \sigma^{-c_{4}''}] \geq \frac{999}{1000}.
\ee
Let $c_{4} = \max(c_{4}', c_{4}'')$ and $C_{4} = \max(C_{4}', C_{4}'')$. Let $L = 100 \lceil C_{4} \sigma^{-c_{4}} \rceil$, and define the intervals $I_{k} = \{ (L-1)k+1, \ldots, Lk \}$ for $k \in \mathbb{N}$. Using the Markov property to apply Inequality \eqref{IneqLeavingIntermodalIneq2}  iteratively over the intervals  $I_{1},\ldots,I_{\lceil \sigma^{-10} \rceil}$, this implies that for $c_{5} = 10 c_{4}$ and some constant $C_{5} = C_{5}(\delta) > 0$,
\be  \label{IneqEscapeSmallSet}
\P[ \tau^{(1)} < C_{5} \sigma^{-c_{5}}], \, \P[ \hat{\tau}^{(1)} < C_{5} \sigma^{-c_{5}}] \geq 1 - e^{-\sigma^{-10}}.
\ee

Combining the bound  \eqref{IneqMixingBulkHat} on the mixing of $\hat{K}_{\sigma}$ on $(-\sigma^{-10},-2\delta)$ with the bound \eqref{IneqVeryLowExcursionsVeryUnlikely} on the possibility of excursions below $-\sigma^{-10}$ and the bound \eqref{IneqEscapeSmallSet} on the length of excursions above $-2 \delta$ implies that 
\be  \label{EqVeryStrongHatMixing2}
\sup_{-\sigma^{-9} < x,y < -2 \delta} \| \hat{K}_{\sigma}^{S}(x,\cdot) -  \hat{K}_{\sigma}^{S}(y,\cdot) \|_{\mathrm{TV}} \leq C_{6} e^{-c_{6} \sigma^{-5}}
\ee
for some constants  $0 < c_{6}, c_{7}, C_{6}, C_{7} < \infty$ that depend only on $\delta$ and running time $S = \lceil C_{7} \sigma^{-c_{7}} \rceil$. Fixing \textit{e.g.} $\delta = 0.01$ and combining this with the bound \eqref{IneqEscapeSmallSet} on the length of excursions above $-2 \delta$ completes the proof of the proposition for all starting points.
 \end{proof}

This completes the proof of the Lemma. 
\end{proof}

Next, we have:

\begin{lemma} 
With notation as above, Parts \textbf{(3-4)} of Assumptions \ref{AssumptionsMeta1} and Parts \textbf{(1-4)} of Assumptions \ref{AssumptionsMeta2} hold.
\end{lemma}

\begin{proof}
 Parts \textbf{(3-4)} of Assumptions \ref{AssumptionsMeta1} and Parts \textbf{(2-3)} of Assumptions \ref{AssumptionsMeta2}  follow almost immediately from the Lyapunov condition \eqref{IneqLyapMulti}. In particular, consider a copy of the chain $\{X_{t}\}_{t \geq 0}$ started at any point $x$. Iteratively applying \eqref{IneqLyapMulti},
\be 
\E[V_{\sigma}(X_{t}) \textbf{1}_{\tau_{x,U}}] \leq \E[(1-\alpha)V_{\sigma}(X_{t-1}) + C_{\sigma}] \leq \ldots \leq (1 - \alpha)^{t} V_{\sigma}(X_{0}) + \frac{C_{\sigma}}{\alpha}. 
\ee 
Applying Markov's inequality with an appropriate choice of $x$ immediately gives Parts \textbf{(3-4)} of Assumptions \ref{AssumptionsMeta1} and Part \textbf{(3)} of Assumptions \ref{AssumptionsMeta2}. Part \textbf{(3)} of Assumptions \ref{AssumptionsMeta2} is of course implied by  \eqref{IneqLyapMulti} directly.

Part \textbf{(1)} of Assumptions \ref{AssumptionsMeta2} holds from the symmetry of the situation and the fact that we have proved Assumption \ref{AssumptionsMeta1} holds.

For Part \textbf{(4)} of Assumption \ref{AssumptionsMeta2}, Inequality \eqref{IneqNP1} is clear. Inequality \eqref{IneqNP2} follows immediately from the single-step Lyapunov condition \eqref{IneqLyapMulti} and Markov's inequality. 
\end{proof}

Since we have verified all the assumptions of Lemma \ref{LemmaMeta1} and \ref{LemmaMeta2}, applying them completes the proof of Theorem \ref{ThmHmcMultimodal}.

\subsection{Proof of Theorem \ref{ThmHmcDegenerate}}

We begin by writing down an exact expression for the transition kernel of interest. Recall that $K_{\sigma}$ is the transition kernel for a Markov chain in $\mathbb{R}$, even though the original problem is about a Markov chain in $\mathbb{R}^{2}$. By a small abuse of notation, write $K_{\sigma}(x,y)$ for the density of $K_{\sigma}(x,\cdot)$ at $y$. Next, fix $p,q \in \mathbb{R}$, and let $(p(t),q(t))$ be a solution to Hamilton's equations with potential $U(p,q) = \frac{1}{2} p^{2} + \frac{1}{2} q^{2}$ and initial conditions $q(0) = q$, $p(0) = p$. $(p(t),q(t))$ solve the coupled equations
\be
q'(t) = p(t), \qquad p'(t) = - q(t).
\ee
These equations are well-known to have solutions of the form
\be
q(t) &= A \cos(t) + B \sin(t) \\
p(t) &= C \cos(t) + D \sin(t);
\ee
applying the initial conditions to compute $A,B,C,D$, we find that the solution at time $t = T_{\sigma}$ is
\be \label{EqExpHarmonic}
q(T_{\sigma}) = q \, \cos(T_{\sigma}) + p \, \sin(T_{\sigma}).\\
\ee
Thus, for fixed $q,y \in \mathbb{R}$ and $0 < T_{\sigma} < 1$, we have
\be 
\{p \in \mathbb{R} \, : \, q(T_{\sigma}) = y \} = \{ \frac{y - q \cos(T_{\sigma})}{\sin(T_{\sigma})} \}.
\ee
Thus, for fixed $x,y \in \mathbb{R}$, we have that
\be  \label{EqExactDensUni}
K_{\sigma}(x,y) = \frac{1}{\sin(T_{\sigma})} \frac{1}{\sqrt{2 \pi}} e^{-\frac{1}{2} (\frac{y - x \cos(T_{\sigma})}{\sin(T_{\sigma})})^{2}} \\
\ee
is the density of a standard normal with mean $x \cos(\sigma)$ and variance $\sin(\sigma)^{2}$.

Next, we prove Inequality \eqref{IneqRelDeg} by calculating $\|K_{\sigma} f \|_{L_2(\pi_{\sigma})}$ for the test function $f(x) = x$. By Equation \eqref{EqExactDensUni}, we can exactly calculate

\be 
(K_{\sigma} f)(x) = \int_{y} K_{\sigma}(x,y) \, y dy = x \cos(\sigma).
\ee

Thus, 
\be
\rho_{\sigma} &\leq 1- \frac{ \| K_{\sigma} f\|_{L_{2}(\pi)}} { \| f \|_{L_{2}(\pi)}} \\
&= 1-\frac{\sqrt{ \int_{x} \frac{x^{2} \cos(T_{\sigma})^{2}}{\sqrt{2 \pi}} e^{-\frac{1}{2} x^{2}} dx}} {\sqrt{ \int_{x} \frac{x^{2} }{\sqrt{2 \pi}} e^{-\frac{1}{2} x^{2}} dx}} \\
&= 1- \cos(T_{\sigma})= T_{\sigma}^{2} + O(T_{\sigma}^{4}). 
\ee
This immediately implies Inequality \eqref{IneqRelDeg}.

Finally, we prove Inequality \eqref{EqMultiCheegDeg}.  By Theorem \ref{thm:1}, 
\be
\Phi_{\sigma} &= \Phi^+ \, \mathbb{E}_\mathbb{Q}\bigg[\frac{1}{N_{\{0\}}} \cdot \mathbbm{1}\{N_{\{0\}} \, \mathrm{odd}\}\bigg] \bigg/ \pi_{\sigma}((-\infty,0)) \\
&=  T_{\sigma} \cdot \int_{\{0\}} f_{\sigma}(q) \mathrm{d}q \cdot \mathbb{E}_\mathbb{Q}\bigg[\frac{1}{N_{\{0\}}} \cdot \mathbbm{1}\{N_{\{ 0\}} \, \mathrm{odd}\}\bigg] \\
&= T_{\sigma}  \mathbb{E}_\mathbb{Q}\bigg[\frac{1}{N_{\{0\}}} \cdot \mathbbm{1}\{N_{\{0\}} \, \mathrm{odd}\}\bigg] \leq \sigma.
\ee
Taking logs, this immediately implies Inequality \eqref{EqMultiCheegDeg}.

\end{document}